\documentclass[11pt]{amsart}
\usepackage{amsmath}
\usepackage{amssymb}
\usepackage{amscd}
\usepackage{url}
\usepackage[latin2]{inputenc}
\usepackage[T1]{fontenc}
\usepackage{pstricks}
\usepackage{pst-node}
\usepackage{pst-coil}
\usepackage{graphicx,color}
\usepackage{hyperref}

\newcommand{\R}{\mathbb{R}}
\newcommand{\E}{\mathbb{E}}
\newcommand{\Var}{\mathrm{Var}}
\newcommand{\PP}{\mathbb{P}}
\newcommand{\tm}{t_{\mathrm{mix}}}
\newcommand{\eps}{\epsilon}
\newcommand{\gas}{\alpha^{\star}}
\newcommand{\bn}{\{0, 1\}^n}
\newcommand{\1}{\mathbf{1}}
\newcommand{\cf}{\bar{f}}
\newcommand{\PL}{\texttt{PLU} }
\newcommand{\Px}{\mathrm{PX}}
\newcommand{\I}{\mathrm{Inf}}

\def\R{{\mathbb R}}
\def\e{{\mathcal{E}}}
\def\E{{\mathbb E}}
\def\H{{\mathcal{H}}}

\theoremstyle{definition}
\newtheorem{example}{Example}[section]
\newtheorem{definition}[example]{Definition}

\theoremstyle{plain}

\newtheorem{lemma}[example]{Lemma}
\newtheorem{proposition}[example]{Proposition}
\newtheorem{theorem}[example]{Theorem}

\newtheorem{corollary}[example]{Corollary}

\theoremstyle{remark}
\newtheorem*{acknowledgements}{Acknowledgements}
\newtheorem{remark}[example]{Remark}
\numberwithin{equation}{section}

\begin{document}
\title[Reverse hypercontractivity]{On reverse hypercontractivity}

\author{Elchanan Mossel}
\address{Elchanan Mossel \\ Statistics and Computer Science\\  367 Evans Hall \\ University of California \\ Berkeley, CA \\ USA  \& Faculty of Mathematics and Computer Science \\ Weizmann Institute \\ Israel}
\email{mossel@stat.berkeley.edu, elchanan.mossel@weizmann.ac.il}
\thanks{The first author is supported by NSF DMS
0548249 (CAREER) and NSF DMS 1106999 awards, by DOD ONR grant N000141110140, by ISF
grant 1300/08 and by a Minerva Grant. Most of this work was conducted when the author was at the Weizmann institute.}
\author{Krzysztof Oleszkiewicz}
\address{Krzysztof Oleszkiewicz \\ Institute of Mathematics \\ University of Warsaw \\ul. Banacha 2, 02-097 \\ Warsaw, Poland}
\email{koles@mimuw.edu.pl}
\thanks{The second author is partially supported by Polish MNiSzW Grant N N201 397437}
\author{Arnab Sen}
\address{Arnab Sen \\ Statistical Laboratory, DPMMS, Wilberforce Road, Cambridge, CB3 0WB, United Kingdom}
\email{a.sen@statslab.cam.ac.uk}

\date{}

\maketitle

\begin{abstract}
We study the notion of reverse hypercontractivity. We show that reverse hypercontractive inequalities are implied
by standard hypercontractive inequalities as well as by the modified log-Sobolev inequality.
Our proof is based on a new comparison lemma for Dirichlet forms and an extension of the Stroock-Varopoulos inequality.

A consequence of our analysis is that {\em all} simple operators $L=Id-\E$ as well as their tensors satisfy  uniform reverse hypercontractive inequalities. That is, for all $q<p<1$ and every positive valued function $f$ for
$t \geq \log \frac{1-q}{1-p}$ we have $\| e^{-tL}f\|_{q} \geq \| f\|_{p}$.
This should be contrasted with the case of hypercontractive inequalities for simple operators where $t$ is known to depend not only on $p$ and $q$ but also on the underlying space.

The new reverse hypercontractive inequalities established here
imply new mixing and isoperimetric results for short random walks
in product spaces, for certain card-shufflings, for Glauber dynamics in high-temperatures spin systems as well as
for queueing processes. The inequalities further imply a quantitative Arrow impossibility theorem for general product distributions and inverse polynomial bounds in the number of players for the non-interactive correlation distillation problem with
$m$-sided dice.

\end{abstract}


\section{Introduction}

\subsection{Background}
Log-Sobolev and hypercontractive inequalities play a fundamental role in a number of areas
in analysis and probability theory  including the study of Gaussian processes (see, e.g.,~\cite{Gross:78,Janson:97}),
analysis of Markov chains (see, e.g.,~\cite{Saloff-Coste:99}) and discrete Fourier analysis starting in~\cite{KaKaLi:88,Talagrand:94}.

One of the first and most useful hypercontractive inequalities
due to Bonami-Nelson-Gross-Beckner~\cite{Bonami:70,Nelson:73,Gross:75,Beckner:75} states that if
$(\Omega,\mu) = (\{0,1\},\frac{1}{2} (\delta_0 + \delta_1))$ then the operator
$T_t = e^{-t(Id - \E_\mu)}$ satisfies
\begin{equation} \label{eq:beckner}
\|T_t f \|_p \leq \| f \|_q, \quad \forall f : \Omega \to \R,\, p > q > 1,\, t \geq \frac{1}{2} \log \frac{p-1}{q-1}.
\end{equation}
In probabilistic language, the operators $(T_t)_{ t \ge 0}$ form a Markov semigroup with $T_t f(x) = \E [ f(X_t) |X_0 =x]$, where  $(X_t)_{ t \ge 0} $ is a  continuous-time Markov chain on $\Omega$ where  the particle jumps from the state $y$ to the state  $z$ with probability $\mu(z)$ and the gaps between successive jumps are distributed as independent exponential random variables.

The strength of a simple hypercontractive inequality  like \eqref{eq:beckner} lies in the fact that it  tensorizes. This led to many applications in discrete Fourier analysis (starting with~\cite{KaKaLi:88})
and even earlier in the study of Gaussian processes.
Extending~(\ref{eq:beckner}) to other spaces turned out to be a non-trivial task.
For the case of the spaces $(\Omega,\mu) = (\{0,1\}, \alpha \delta_0 + (1-\alpha) \delta_1)$, with $\alpha \leq 1/2$,
the first bounds were established by~Talagrand~\cite{Talagrand:94}.
Exact formulas have been obtained
by Oleszkiewicz~\cite{Oleszkiewicz:03} in the cases where either $p>2=q$ or $p=2>q>1$.
Wolff then extended these results \cite{Wolff:07} to general discrete spaces and, in a slightly less precise form,
to all $p>q \geq 2$ and all $2 \geq p>q>1$: let $(\Omega,\mu)$ be a finite probability space with
$\alpha = \min_{\omega \in \Omega} \mu \{\omega\}>0$; then there exists some universal positive constant $\varepsilon$ such that
for $p, q$ as above and certain $t_{0}=t_{0}(p,q,\alpha)$, given by an explicit though complicated formula,
\[
t \geq t_{0} \implies \forall f : \Omega \to \R,\,\,\, \| T_t f  \|_p \leq \| f \|_q
\implies t \geq t_{0}-\varepsilon.
\]
Moreover, $\lim_{\alpha \to 0^{+}} t_{0}(p,q,\alpha)=\infty$. This dependency on the smallest atom in space is present in many applications of hypercontractivity starting with~\cite{Talagrand:94}. We note further that the same dependency is arrived at using the exact calculation of the log-Sobolev constant
of simple operators (that is,   operators of the form $T_t = e^{-t(Id-\E_\mu)}$ acting on the function defined on a probability space $(\Omega, \mu)$)  by Diaconis and Saloff-Coste~\cite{DiaconisSaloff-Coste:96}.

A `reverse' hypercontractivity is shortly proved and discussed in a paper by Borell~\cite{Borell82} in the 80's.
This result, proven for the measure $(\Omega,\mu) = (\{0,1\},\frac{1}{2} (\delta_0 + \delta_1))$, states that
\begin{equation} \label{eq:borell}
\|T_t f \|_q \geq \| f \|_p, \quad \forall f : \Omega \to \R_{+},\, 1>p > q ,\, t \geq \frac{1}{2} \log \frac{1 - q}{1-p}.
\end{equation}
This inequality which also tensorizes is indeed `reverse' in many ways. Not only the inequality goes `the other way' and
the roles of $p$ and $q$ get reversed, it is also the case that $p$ and $q$ are less than $1$ (indeed they may be negative(!);  note, however, that the function $f$ has to take
positive values).

As far as we know, Borell's result was first used in a paper published more than 20 years later~\cite{Mossel06}, where it is used to analyze mixing of short random walks on the discrete cube $\{0,1\}^n$ as well as to provide tight bounds on the Non-Interactive Correlation Distillation (NICD) problem. 

Motivated by generalization of applications in~\cite{Mossel06} as well as by other applications that will be discussed later, we wish to extend Borell's results to other discrete probability spaces.
Noting the similarity of the inequalities~(\ref{eq:beckner}) and~(\ref{eq:borell}) it is tempting to conjecture
(as the first named author have done) that the formulas for hypercontracitivity and reverse hypercontractivity are `the same': in particular, for discrete spaces there is a dependency on the size of the smallest atom in space as in the above-mentioned results for hypercontractivity. The conjecture is further supported by the fact that for diffusions both hypercontractivity and reverse hypercontractivity are equivalent to the standard log-Sobolev inequality (for more details see~\cite{Bakry:94}; some pioneering results relating hypercontractivity to reverse hypercontractivity were obtained already in~\cite{BorellJanson82}). 

The conjecture turns out to be far from true. 
In fact our results show that for every discrete probability space
$(\Omega,\mu)$:
\begin{equation} \label{eq:main}
\|T_t f \|_q \geq \| f \|_p, \quad \forall f : \Omega \to \R_{+},\, 1>p > q ,\, t \geq  \log \frac{1- q}{1-p}.
\end{equation}
In particular, reverse hypercontractive inequalities hold uniformly for all probability spaces.

It is well known that hypercontractive inequalities are intimately related to logarithmic  Sobolev inequalities and our proof of \eqref{eq:main}
 is based on extension of this connection to `norms' $p<1$  (such extensions were noted before, see, e.g.,~Bakry's lecture notes~\cite{Bakry:94}). 
At the heart of the proof is a new monotonicity result showing that under
the appropriate normalization log-Sobolev inequalities are monotone in the norm parameter $p$ for all $p \in [0,2]$.
This result in turn is based on an extension of the Stroock-Varopoulos inequality to general norms. The result allows us to show how reverse hypercontractive inequalities follow directly from
standard hypercontractive inequalities and furthermore from standard log-Sobolev and modified log-Sobolev inequalities.

After we develop the theory of reverse hypercontractive inequalities, we derive a number of novel results regarding mixing of Markov chains run for short time starting from large sets, in general cubes, the symmetric group and Ising configurations (via Glauber dynamics). We further derive a quantitative  Arrow's Theorem for general distributions and  inverse polynomial bounds for the NICD problem for general $m$-sided dice.
We proceed with formal definitions and statements of the main results.

\subsection{General setup} \label{setup}
We now turn to the general mathematical setup of the paper.
Let $(\Omega, \mu)$ be a finite probability space (with a natural $\sigma$-field of all subsets of $\Omega$).
We assume $\mu\{\omega\}>0$ for $\omega \in \Omega$.
Let $\E$ denote the expectation operator:
$\E f=\int_{\Omega} f\,d\mu$. We also use the standard notation for the {\em variance} of $f$, $\Var(f)=\E f^{2}-(\E f)^{2}$, and the {\em entropy} of $f>0$, $Ent(f)=\E(f \log f)-\E f \cdot \log \E f$.
Let $\H$ be the space of real-valued functions on $\Omega$. 
Let $\H_{(0,\infty)}$ denote the positive functions.
Let $L:\H \rightarrow \H$ be a linear operator such that
\begin{itemize}
\item $L1=0$ and
\item $L$ is self-adjoint with respect to the $L^{2}(\Omega, \mu)$ structure,
i.e., $\E fLg=\E gLf$ for any $f,g \in \H$, and
\item
$L$ is positive semidefinite, i.e., $\E fLf \geq 0$ for all $f \in \H$, and
\item for any $f \in \H$, and any $\omega \in \Omega$ such that $f \leq f(\omega)$ on $\Omega$, there is $(Lf)(\omega) \geq 0$. 
\end{itemize}
Alternatively, one can replace the fourth condition by the non-negativeness of the carr\'e du champ
(as a function-valued quadratic form), i.e., $L(f^{2}) \leq 2fLf$ for $f \in \H$.
The {\em Markov semigroup} of operators $(T_{t})_{t \geq 0}: \H \rightarrow \H$ {\em generated by $L$} is given by
\[
T_{t}f=e^{-tL}f,
\]
with $T_{0}f=f$ and $\frac{d}{dt}T_{t}f=-LT_{t}f=-T_{t}Lf$. The {\em Dirichlet form} $\e: \H \times \H \rightarrow \R$ {\em associated with $L$}
is given by
\[
\e(f,g)=\E(fLg)=\E(gLf)=\e(g,f)=
-\frac{d}{dt} \E fT_{t}g\Big|_{t=0}.
\]
Recall that in this setup we have
$T_{t}1=1$ for $t \geq 0$ and $\e(f,1)=0$ for $f \in \H$.
The operators $T_{t}$ are symmetric linear contractions in $L^{p}$-norm for every $p \in [1,\infty)$ and $t \geq 0$. They are mean-preserving, i.e.,
\[
\E T_{t}f=\E 1T_{t}f=\E fT_{t}1=\E f,
\]
and positivity preserving, i.e., $T_{t}f \geq 0$ for
$f \in \H$ and $f\ge 0$, thus also order preserving ($f \geq g$ implies $T_{t}f \geq T_{t}g$).
In fact, they preserve also strict positivity: $f>0$ implies $T_{t}f>0$ for $t \geq 0$.
The positivity preserving property implies that $\e(|f|,|f|) \leq \e(f,f)$ for any $f \in \H$.
One can associate to a Markov semigroup $(T_{t})_{t \geq 0}$ on $(\Omega, \mu)$
a time-homogenous reversible $\Omega$-valued Markov process $(X_{t})_{t \geq 0}$
with $X_{t} \sim \mu$ for $t \geq 0$, where the transition probabilities of the process can be read from the formula
$\PP(X_{t_{1}}=\omega_{1}, X_{t_{2}}=\omega_{2})=\E(1_{\omega_{1}}T_{|t_{1}-t_{2}|}1_{\omega_{2}})$
for $\omega_{1}, \omega_{2} \in \Omega$, and, conversely, $(T_t f)(\omega)=\E \left(f(X_{t})|X_{0}=\omega\right)$
for any $f \in \H$ and $\omega \in \Omega$.

\begin{definition}
For $f \in \H$ and $p>0$ we denote by $\| f\|_{p}$ the $p$-{\em th norm} of $f$: $(\E |f|^{p})^{1/p}$.
We extend the definition to $p \in \R$ and $f \in \H_{(0,\infty)}$
by setting $\| f\|_{p}=(\E f^{p})^{1/p}$ for $p \neq 0$, and $\| f\|_{0}=\exp(\E \log f)$.
\end{definition}
Recall that $\| \cdot \|_{p}$ is a true norm for $p \geq 1$  but it is only a pseudo-norm (triangle inequality fails) for $p < 1$ (unless $|\Omega|=1$).
It is an easy and well-known fact that for any $f \in \H_{(0,\infty)}$ the map $p \mapsto \| f\|_{p}$ is continuous and non-decreasing.

Following Borell \cite{Borell82} we extend the definition of duality to $p \in \R$:
\begin{definition}
For a number $p \in \R \setminus \{ 0,1\}$ we define its (H\"older) conjugate $p'=p/(p-1)$, so that $\frac{1}{p}+\frac{1}{p'}=1$. We also set $0'=0$.
\end{definition}

Note that the map $p \mapsto p'$ is a continuous order-reversing involution on $(-\infty,1)$ and $(1,\infty)$ with fixed points $0$ and $2$.
It is worth observing that $(2-p)'=2-p'$ for $p \neq 1$, even though we will not make use of this fact.

\subsection{Log-Sobolev inequalities}
We now recall the definition of log-Sobolev inequalities.

\begin{definition}
For $p \in \R \setminus \{ 0,1 \}$ we say that $p$-logSob is satisfied with constant $C>0$ if
\begin{equation}\label{def:logSob}
Ent(f^{p}) \leq \frac{Cp^{2}}{4(p-1)} \e(f^{p-1},f)
\end{equation}
for every $f \in \H_{(0,\infty)}$. We will say that $1$-logSob is satisfied with constant $C>0$ if
\[
Ent(f) \leq \frac{C}{4} \e(f,\log f)
\]
for $f \in \H_{(0,\infty)}$. Finally, we will say that
$0$-logSob is satisfied with constant $C>0$ if
\[
\Var(\log f) \leq -\frac{C}{2}\e(f,1/f)
\]
for every $f \in \H_{(0,\infty)}$.
\end{definition}

\begin{remark}
Obviously, the cases $p=0$ and $p=1$ of the $p$-logSob inequality are limit cases of the $p$-logSob for
$p \in \R \setminus \{ 0,1\}$ (with the same $C$). 

\end{remark}

\begin{remark}
Logarithmic Sobolev inequalities were introduced by Gross in his seminal paper \cite{Gross:75}. Gross defined logarithmic Sobolev inequality for $p>1$. The definition was later extended by Bakry \cite{Bakry:94} to any real $p$ (including $1$-logSob inequality). 
Finally, we remark that $1$-logSob inequality is also known in the literature as modified log-Sobolev inequality (see, e.g.,~\cite{Wu00, Quastel03,  Goel04,  BobkovTetali06}). The 1-logSob inequality is also called ``entropic inequality'' as it implies  the exponential decay of entropy along the semigroup.
\end{remark}

\begin{remark}
Our definition uses a novel and non-standard normalization factor $\frac{p^{2}}{4(p-1)}$  in \eqref{def:logSob}. 
The choice of this normalization makes our $p$-logSob constants invariant under H\"{o}lder conjugation (see Lemma~\ref{lsdual}).  Moreover, this normalization is crucial to prove  the main result of the paper - 
the monotonicity of the inequality for  $p \in [0, 2]$ (see Theorem~\ref{monotone}).  
\end{remark}

In our main result we prove a  general result relating $p$-logSob inequalities for different values of $p$.

\begin{theorem} \label{monotone}
Let $0 \leq q \leq p \leq 2$. Assume that $p$-logSob holds with a constant $C>0$. Then also $q$-logSob holds true
with the same constant $C$.
\end{theorem}

\begin{remark}
It has been proved in \cite[Proposition 3.1]{Bakry:94} that if $2$-logSob holds with constant $C$,
then any $p$-logSob also holds with the same constant and for $p>0$ the converse is true in case of
diffusions with invariant measure $\mu$. 
\end{remark}

Using the fact that simple operators satisfy $1$-logSob with the constant $4$ (proved in  \cite{BobkovTetali06}; the proof is reproduced in our Lemma~\ref{simple} below)
we obtain the following corollary:
\begin{corollary} \label{cor_simple}
Assume that the semigroup $(T_{t})_{t \geq 0}$ is generated by $L=Id-\E$ or by a tensor of simple operators.
 Then it satisfies the $r$-logSob inequality with constant $4$ for all $r \in [0,1]$.
\end{corollary}

\subsection{Reverse hypercontractive estimates}
Using Theorem~\ref{monotone} we derive the following general reverse hypercontractive bounds:

\begin{theorem} \label{thm:1ls->revhyp}
If a symmetric Markov semigroup $(T_{t})_{t \geq 0}$ satisfies $r$-logSob with constant $C$ and $r \geq 1$ then for all
$q<p<1$ and every $f \in \H_{(0,\infty)}$ for all
$t \geq \frac{C}{4}\log \frac{1-q}{1-p}$ we have
$\| T_{t}f\|_{q} \geq \| f\|_{p}$.
\end{theorem}

Using Corollary~\ref{cor_simple} this implies in turn that:

\begin{corollary} \label{simrevhyp}
If a symmetric Markov semigroup $(T_{t})_{t \geq 0}$ has a simple generator $L=Id-\E$ or it is a tensor product of such simple semigroups,
then for all $q<p<1$ and every $f \in \H_{(0,\infty)}$ for $t \geq \log \frac{1-q}{1-p}$ we have
$\| T_{t}f\|_{q} \geq \| f\|_{p}$.
\end{corollary}

In fact, for simple operators we derive the following stronger result:

\begin{theorem} \label{simrevhyp_strong}
Assume that a symmetric Markov semigroup
$(T_{t})_{t \geq 0}$ has a simple generator $L=Id-\E$ or it is a tensor product of such simple semigroups. Let
$f \in \H$ be strictly positive.
Then for all $q<p \leq 0$ and $t \geq \log \frac{2-q}{2-p}$, and also for all $0 \leq q<p<1$ and
$t \geq \log \frac{(1-q)(2-p)}{(1-p)(2-q)}$
we have
$\| T_{t}f\|_{q} \geq \| f\|_{p}$.
\end{theorem}

\begin{remark}
The reverse hypercontractive inequality~(\ref{eq:borell}) for a simple
operator and the space $\{-1,1\}$ with the uniform measure was derived prior to ours by Borell.
His result is tight.
\end{remark}

\subsection{Application 1: Mixing of large sets in Markov chains}

Our first application of the new inequality is to mixing of Markov chains from  large sets.
The statement and proof of the theorem below are a generalization of the main result of~\cite{Mossel06} where it was proven for the random walk on the discrete cube $\{0,1\}^n$.

\begin{theorem} \label{t:intersection}
Let $(X_t)_{ t \ge 0}$ be a continuous-time
Markov chain on a finite state space $\Omega$ which is reversible with respect to the invariant probability measure $\pi$.
Let $(T_{t})_{t \geq 0}$ be the semigroup defined by $T_t f (x)= \E^x f(X_t)$ for $f \in \mathcal H$.
Assume $(T_{t})_{t \geq 0}$ satisfies $1$-logSob with constant $C$. Let $a,b \geq 0$ and
let $A, B \subseteq \Omega$ with $\pi\{A\} = \exp(-a^2/2)$ and $\pi\{B\}  = \exp(-b^2/2)$.  Let $X_0$ be distributed according to $\pi$. Then
\begin{equation}\label{ieq:twosetAB}
\PP\{ X_0 \in A, X_t \in B \} \ge \exp \left( - \frac{1}{2} \frac{a^2 + 2 e^{-2t/C}ab + b^2 }{1- e^{-4t/C}} \right).
\end{equation}
\end{theorem}

This theorem should be compared to the two main techniques for proving lower bounds on $\PP\{ X_0 \in A, X_t \in B \}$.
\begin{itemize}
\item
First, the {\em Expander Mixing Lemma} (see, e.g.,~\cite[Chapter 9]{AlonSpencer}) implies that if Poincar\'e inequality holds with constant $D$ then:
\begin{equation} \label{eq:expander_mixing}
\PP\{ X_0 \in A, X_t \in B \} \geq \pi\{A\} \pi\{B\}  - \sqrt{\pi\{A\} \pi\{B\} } e^{-t/D}.
\end{equation}
The inequality~\eqref{eq:expander_mixing} will be better (up to constants) than our inequality~(\ref{ieq:twosetAB}) in the case where the sets
$A$ and $B$ are large,
say  $\pi\{A\} \pi(B) \geq 2 \sqrt{\pi\{A\} \pi\{B\} } e^{-t/D}$, since in this case we obtain the lower bound of $\frac{1}{2}\pi\{A\} \pi\{B\} $ which is
(except for the factor $2$) the best that one can hope for. However, in the case where the sets $A$ and $B$ are small, say
$\pi\{A\} \pi\{B\}  \leq \sqrt{\pi\{A\} \pi\{B\} } e^{-t/D}$, the expander mixing lemma gives nothing while~(\ref{ieq:twosetAB}) gives a lower bound
that is a power of the measures of the original sets.
\item
The second technique uses total variation mixing times. Indeed, if the worst total variation distance at time $t$ is at most $\eps$, then
we have:
\begin{equation} \label{eq:mixing_time}
\PP\{ X_0 \in A, X_t \in B \} \geq \pi\{A\} (\pi\{B\} - \eps).
\end{equation}
Again - applying this bound requires that one of the sets $A$ or $B$ is large (of measure at least $\eps$).
Moreover, in many examples the time $t$ when the total variation distance
is at most $1/e$ is much larger than the $1$-logSob constant $C$. Therefore if $t$ is of order $C$, then the mixing time bound~(\ref{eq:mixing_time})
gives nothing while our result~(\ref{ieq:twosetAB}) gives an efficient lower bound.
\end{itemize}

We demonstrate this point by proving new mixing bounds from large sets for various classical Markov chains, including:
\begin{itemize}
\item
{\em Short random walks on general product spaces}. 
In this case we derive tighter results in subsection~\ref{subsec:short}. 
\item
{\em Glauber dynamics on Ising model on finite boxes}. The results of \cite{Stroock92a, Stroock92b, Zegar92, Martinelli94a, Martinelli94b} imply that in ``high temperatures'' there is a uniform bound on  for $2$-logSob constant in the box $[-n,n]^d$ while the mixing time of the Glauber dynamics is $\tm = \Theta(\log n)$. 
Thus our results provide new bounds for mixing of big sets in this setup.  Details are provided in subsection~\ref{subset:ising}.
\item
{\em The random transposition card shuffle on the symmetric group}. Here
it is known that hat the 1-logSob constant $C$ of this chain is of order $n$~\cite{Quastel03, BT03, Goel04} while the mixing time is $\Theta(n \log n)$. Thus again, we obtain new results on mixing from large sets. 
Similar logic applies to the {\em Top-to-random transposition walk on symmetric group}, see~\cite{Goel04, Diaconis92} for the 1-logSob constant and the mixing time.  We provide the details in subsections~\ref{subsec:transpose_walk} and \ref{subsec:top_to_random}.
\item  {\em Random walk on the spanning trees of certain graphs}. See subsection~\ref{subsec:spanning}.
\item {\em The Bernoulli-Laplace model.} See subsection~\ref{subsec:BL}.

\item
A natural Markovian queueing process - the $q/q/\infty$ Markov process. The last example is interesting since it has infinite $2$-logSob constant
and an infinite mixing time. More details on this example are given in subsection~\ref{subsec:queue}
\end{itemize}

\subsection{Application 2: A general quantitative Arrow theorem}
Arrow's Impossibility Theorem~\cite{Arrow:50,Arrow:63} is a fundamental result in  social choice theory. 
It considers $n$ voters who rank $k$ candidates. 
Arrow considered functions $F : S_k^n \to \{-1,1\}^{k \choose 2}$ that aggregate individual rankings (elements of the permutation group $S_k$) to result in a preference between every pair of the $k$ alternatives.
Arrow showed that if the following desired properties hold simultaneously when $k \geq 3$:
\begin{itemize}
\item Transitivity - $F(\sigma)$ induces a transitive ranking for all $\sigma \in S_k^n$,
\item Unanimity  - for every pair of alternatives $a$ and $b$, if all voters rank $a$ above $b$ then $F$
also ranks $a$ above $b$,
\item Independence of irrelevant alternatives (IIA) - for every pair of alternatives, the resulting outcome regarding the preference between $a$ and $b$ is determined by the individual preferences between $a$ and $b$,
\end{itemize}
then $F$ is a dictator function, i.e., it is determined by a single voter.

It is natural to ask how robust is the result when considering natural distributions over $S_k^n$.
This question was analyzed by Kalai~\cite{Kalai:02} who studied it for the case of the uniform distribution over $S_3^n$ and showed that for every $\eps > 0$, there exists a $\delta > 0$ such that if $F$ satisfies:
\begin{itemize}
\item $\delta$-Transitivity: $\PP\{F(\sigma) \mbox{ is transitive}\} \geq 1-\delta$,
\item Fairness:   for every pair of alternatives $a$ and $b$, $\PP \{F \mbox{ ranks $a$ above $b$}\} = 1/2$,
\item IIA,
\end{itemize}
then there exists a dictator function $G$ such that $\PP\{F(\sigma) \neq G(\sigma)\}\leq \eps$.

Following a challenge by Kalai, Mossel~\cite{Mossel11} proved a stronger result for any number of alternatives and without the assumption that $F$ is fair.  His result shows that for $k \geq 3$ and every $\eps > 0$, there exists a $\delta > 0$ such that if $F$ satisfies
\begin{itemize}
\item $\delta$-Transitivity:  $\PP \{F(\sigma) \mbox{ is transitive}\}\geq 1-\delta$,
\item IIA,
\end{itemize}
then there exists a function $G$, which is transitive and satisfies the IIA property, such that $\PP\{F(\sigma) \neq G(\sigma)\}\leq \eps$. A complete characterization of all functions $G$ that are IIA and transitive is given by Wilson~\cite{Wilson:72} - these functions include dictators, functions taking two values etc.

A key ingredient of the proof in~\cite{Mossel11} is the use of reverse hypercontractive inequalities.
It is further noted in~\cite{Mossel11} that it should be possible to extend the proof to general product distributions on $S_k^n$ given appropriate reverse hypercontractive bounds for general two point spaces. Our results imply the following extension. 

\begin{theorem}[Quantitative Arrow's theorem for general distribution] \label{thm:maingeneralintro}
Let $\varrho$ be general distribution on $S_k$ with $\varrho$ assigning positive probability to each element of $S_k$. Let $\PP$ denote the distribution $\varrho^{\otimes n}$ on $S_k^n$. Then
for any number of alternatives $k \geq 3$ and $\eps > 0$,
there exists $\delta = \delta(\eps, \rho)>0$, such that for every $n$, if $F : S_k^n \to \{-1,1\}^{k \choose 2}$
satisfies
\begin{itemize}
\item
IIA and
\item
$\PP \{F(\sigma) \mbox{ is transitive}\} \geq 1-\delta$.
\end{itemize}
Then there exists a function $G$ which is transitive and satisfies the IIA property and
$\PP\{F(\sigma) \neq G(\sigma)\} \leq \eps$
\end{theorem}

We note that considering general product distributions gives a more realistic model of actual voting (though the independence assumption in this line of work is still problematic in real voting scenarios). 

\subsection{Application 3: Non-interactive correlation distillation from dice source}
The problem of non-interactive correlation distillation deals with players who receive correlated
random strings and whose collective goal is to agree with the highest possible probability on a random variable
with a given distribution.
Suppose there are $k \ge 2$ players and a `cosmic source'.
Assume first that the  source generates a  string $x$ of  $n$ i.i.d.\ bits.
Each player gets to receive an independent noisy copy of $x$.  Each player then produces a single random bit based on her input. The players wish to have unanimous agreement on their outputs but are not allowed to communicate.
The problem is to understand to what extent the players can successfully `distill' the correlations in their strings into a shared random bit.

 This problem has been considered in \cite{Alon91, Mossel05, Yang04, Mossel06} when the input strings consist of i.i.d.\  fair coin flips. Here we consider a more general version of the problem where each bit is an outcome of a throw  of a fair dice with $m$-faces, $m \ge 2$. Let us introduce some notations. Let  $\Omega = \{ 1, 2, \ldots, m\}$ denote the set of the possible outcomes of a dice. Let $x = (x_1, x_2, \ldots, x_n)$ be  a random vector consisting of $n$ i.i.d.\ random variables, each being uniformly distributed  over $\Omega$. Fix $\rho\in [0,  1)$. Let $y$ be a $\rho$-correlated copy of $x$ (that is, for each $j \leq n$ independently, with probability $\rho$, $y_{j}=x_{j}$, and with probability $(1-\rho)$, $y_{j}=x_{j}'$, $x'$ being an independent copy of $x$) 
and let $(y^i)_{1 \le i \le k}$ be conditionally independent copies of $y$ given $x$.  Let player $i$ use the function $F_i : \Omega^n \to \Omega$ to produce her output  $F_i(y^i)$.  The functions $F_i$ are all assumed to be balanced, that is, $ \PP\{F_i(x) = j\} = m^{-1}$ for all $i, j$.

Define
\[ \mathcal {M}_\rho(k, n) = \sup_{ (F_i)_{1 \le i \le k}} \PP\{ \text{all players output the same bit} \}, \]
 where the supremum is taken over all choices of balanced functions $ (F_i)_{1 \le i \le k}$. Since a balanced function defined on $n$  variables can be also thought of a balanced function of $(n+1)$ variables, for a fixed $k$ and $\rho$, $\mathcal {M}_\rho(k, n)$ is a non-decreasing function of $n$ and so, $\lim_{n \to \infty } \mathcal{M}_\rho(k, n)$ exists.  One of the main results of \cite{Mossel06} says that when $m = 2$, we have
 \[ \lim_{n \to \infty } \mathcal{M}_\rho(k, n) = k^{ - \frac{1}{\rho^2} +1 +o(1)} \quad \text{ as } k \to \infty. \]
 The upper bound of the above result uses an application of reverse hypercontractivity for simple semigroup on symmetric two-point space.  Here we generalize this bound and give an  inverse polynomial bounds (in $k$) on the agreement probability  for general $m$.
\begin{theorem} \label{thm:nicd}
Fix $\rho \in (0, 1)$. Then there exist positive constants $\gamma_1=\gamma_1(\rho),\gamma_2 = \gamma_2(\rho), c_1 = c_1(m,\rho)$ and $c_2=c_2(m,\rho)$ such that for all $k \ge 2$,
\[ c_2 k^{  - \gamma_2} \le \lim_{n \to \infty } \mathcal{M}_\rho(k, n) \le c_1k^{-\gamma_1}.\]
\end{theorem}

\begin{acknowledgements}
The first named author enjoyed the hospitality of Isaac Newton Institute, Cambridge while completing part of  this research.
The second named author enjoyed hospitality of University of California, Berkeley and Isaac Newton Institute, Cambridge while doing this research.
We thank Dominique Bakry, Franck Barthe, Nick Crawford, Michel Ledoux and Cyril Roberto for helpful comments and discussions. We thank an anonymous referee for numerous helpful suggestions including suggesting simpler proof of Lemma 2.3. 
\end{acknowledgements}

\section{Comparison of Dirichlet forms}

The following theorem extends the classical Stroock-Varopoulos inequality \cite{Stroock84, Varopoulos85}  (covering the case $p=2$,
$q \in (1,2]$ of the present result). Theorem~\ref{sv} is the main tool in proving Theorem~\ref{monotone}.
Note that some terms in the statement below may take negative values.

\begin{theorem} \label{sv}
Let $p,q \in (0,2]\setminus \{ 1\}$ and $p>q$. Then
\[
qq'\e(g^{1/q},g^{1/q'}) \geq pp'\e(g^{1/p},g^{1/p'})
\]
for every $g \in \H_{(0, \infty)}$.
\end{theorem}

\begin{remark} \label{case1}
The above result has a natural extension to the case $p=1$ or
$q=1$, with $\e(\log g, g)$ replacing the right (resp. left) hand side of the asserted inequality;
then it suffices to use functions $\varphi_{1}(x)=\log x$ and $\varphi_{2}(x)=x$ (or $\psi_{1}(x)=\log x$
and $\psi_{2}(x)=x$, respectively) in the proof. One can also simply pass to the limit.
\end{remark}

The proof of the theorem will use the following lemmas.

\begin{lemma} \label{comp}
Let $I$ be a non-empty convex subset of $\R$. Assume that
some functions $\varphi_{1}, \varphi_{2}, \psi_{1}, \psi_{2}:
I \rightarrow \R$ satisfy
\[
(\varphi_{1}(a)-\varphi_{1}(b))
(\varphi_{2}(a)-\varphi_{2}(b)) \leq
(\psi_{1}(a)-\psi_{1}(b))
(\psi_{2}(a)-\psi_{2}(b))
\]
for all $a, b \in I$. Then for every $f:\Omega \rightarrow I$
there is
\[
\e(\varphi_{1}(f),\varphi_{2}(f)) \leq
\e(\psi_{1}(f),\psi_{2}(f)),
\]
where by $\varphi_{1}(f)$ we denote
$\varphi_{1} \circ f \in \H$, etc.
\end{lemma}

\begin{proof}

The proof follows from the identity 
\begin{equation}\label{eq:dir_identity}
\e(F,G)=-\frac{1}{2}\sum_{x, y \in \Omega}\E 1_{x}L1_{y} \cdot \left(F(x)-F(y)\right)\left(G(x)-G(y)\right)
\end{equation}
which holds for all $F, G \in \H$. Note that for $x \neq y$ there is $\E 1_{x}L1_{y} \leq 0$. Indeed, 
$L(-1_{y}) \geq 0$ on $\Omega \setminus \{ y\}$ since the function $-1_{y}$ attains its global maximum
at all $x$'s different from $y$.

\end{proof}

\begin{lemma} \label{comp2}
Assume that $I$ is a convex non-empty subset of $\R$ and functions
$\varphi_{1},$ $\varphi_{2},$ $\psi_{1},$ $\psi_{2}:I \rightarrow \R$ are differentiable and such that
\[
\varphi_{1}'(a)\varphi_{2}'(b)+\varphi_{1}'(b)\varphi_{2}'(a)
\leq
\psi_{1}'(a)\psi_{2}'(b)+\psi_{1}'(b)\psi_{2}'(a)
\]
for all $a, b \in I$. Then for every $f:\Omega \rightarrow I$
there is
\[
\e(\varphi_{1}(f),\varphi_{2}(f)) \leq
\e(\psi_{1}(f),\psi_{2}(f)).
\]
\end{lemma}

\begin{proof}
By Lemma \ref{comp} it suffices to check whether
$\Phi:I \times I \rightarrow \R$ given by
\[
\Phi(a,b)=
(\psi_{1}(a)-\psi_{1}(b))
(\psi_{2}(a)-\psi_{2}(b))-
(\varphi_{1}(a)-\varphi_{1}(b))
(\varphi_{2}(a)-\varphi_{2}(b))
\]
is nonnegative. Clearly, $\Phi(x,x)=0$ and
$\frac{\partial \Phi}{\partial a}(x,x)=0$ for all $x \in I$.
Now it is enough to notice that the assumptions of Lemma \ref{comp2} yield $\frac{\partial}{\partial b}
\frac{\partial}{\partial a}\Phi \leq 0$ which implies
that $\Phi(\cdot,x)$ is non-decreasing on $[x,\infty) \cap I$ and non-increasing on $(-\infty,x] \cap I$.
\end{proof}

We are now ready to prove Theorem~\ref{sv}.

\begin{proof}
It suffices to use Lemma \ref{comp2} with $I=(0,\infty)$,
$\varphi_{1}(x)=px^{1/p}$, $\varphi_{2}(x)=p'x^{1/p'}$,
$\psi_{1}(x)=qx^{1/q}$, and $\psi_{2}(x)=q'x^{1/q'}$.
Indeed, to verify the assumptions of Lemma \ref{comp2} one
needs to check whether for all $a,b>0$, 
%
\[
 (a/b)^{\frac1p -\frac12}+ (a/b)^{-(\frac1p -\frac12)} \leq  (a/b)^{\frac1q -\frac12}+   (a/b)^{-(\frac1q -\frac12)}.
\]
This is, however, obvious since the function $w \mapsto s^{w}+s^{-w}$ is even and convex for every $s>0$ and thus it is non-decreasing on $(0,\infty)$.
Choosing $s=a/b$ and recalling that $0 < 1/p - 1/2 \leq 1/q - 1/2$ ends the proof.
\end{proof}


We further obtain the following.

\begin{corollary} \label{0ls-poin}
For any  $f \in \H_{(0, \infty)}$ there is
\[
\e(\log f, \log f) \leq -\e(f,1/f).
\]
\end{corollary}

\begin{proof}
It follows immediately from Lemma \ref{comp2} applied to
$I=(0,\infty)$ with $\varphi_{1}(x)=\log x$, $\varphi_{2}(x)=\log x$, $\psi_{1}(x)=x$, and $\psi_{2}(x)=-1/x$.
\end{proof}

\begin{remark}
For any positive $g$, the function $u \mapsto \frac{1}{u(1-u)}\e(g^{u},g^{1-u})$ is log-convex on the real line:
either it is positive and its logarithm is convex, or it is identically equal to zero (if $g$ is constant).
It is also obviously symmetric with respect to $1/2$, so that it is non-decreasing on $[1/2, \infty)$, which is a re-formulation of Theorem 2.1. Indeed, the log-convexity follows easily from the formula \eqref{eq:dir_identity} and H\"older's inequality, if one can first prove that the functions $u \mapsto (b^{u}-a^{u})/u$ and, equivalently, $u \mapsto (b^{1-u}-a^{1-u})/(1-u)$ are log-convex for any pair of fixed nonnegative numbers $a>b$. This, however, is an immediate consequence of the identity 
$
(b^{u}-a^{u})/u=\int_{a}^{b} s^{u-1}\,ds,
$
and H\"older's inequality. We skip standard discussion of the cases $u=0$ and $u=1$.
\end{remark}

\section{Logarithmic Sobolev inequalities}

In this section we prove various properties of log-Sobolev inequalities and in particular Theorem~\ref{monotone}.
We begin with a simple claim relating $0$-logSob to the Poincar\'e inequality. We suspect  that both Lemma~\ref{0ls=poin} and Lemma~\ref{lsdual} below were previously known in the literature but we did not find any explicit reference. 

\begin{lemma} \label{0ls=poin}
$0$-logSob holds with constant $C$ if and only if the standard Poincar\'e inequality holds with constant $C/2$, i.e.,
\[
\Var(g) \leq \frac{C}{2} \e(g,g)
\]
for every $g \in \H$.
\end{lemma}

\begin{proof}
For $g \in \H$ and $\delta>0$ set $f=e^{\delta g}$. Assuming that $f$ satisfies $0$-logSob with constant $C$ we obtain
\[
\Var(\delta g) \leq -\frac{C}{2}\e(e^{\delta g}, e^{-\delta g}).
\]
By the homogeneity of variance and bilinearity of $\e$ we get
\[
\Var(g) \leq \frac{C}{2}\e \Big(\frac{e^{\delta g}-1}{\delta}, \frac{1-e^{-\delta g}}{\delta} \Big) \stackrel{\delta \to 0^{+}}{\longrightarrow} \frac{C}{2}\e(g,g).
\]
On the other hand, let $f \in \H$ be positive and assume that the Poincar\'e inequality holds with constant $C/2$.
By using it for $g=\log f$ we arrive at
\[
\Var(\log f) \leq \frac{C}{2}\e(\log f, \log f) \leq
-\frac{C}{2}\e(f,1/f),
\]
where we have used Corollary \ref{0ls-poin}.
\end{proof}

The following easy observation allows us to restrict study of the $p$-logSob inequalities to the case $p \in [0,2]$ (also, it reveals that $1$-logSob is, in a sense, a replacement for $\pm \infty$-logSob).

\begin{lemma} \label{lsdual}
If $p \in \R \setminus \{ 1\}$ and $p$-logSob is satisfied with a constant $C$ then also $p'$-logSob holds, with the same constant.
\end{lemma}

\begin{proof}
For $p=0$ there is nothing to prove, whereas for $p \neq 0$
it suffices to notice that by setting $g=f^{p}$ we obtain an equivalent `self-dual' version of $p$-logSob:
\begin{equation} \label{self-dual}
Ent(g) \leq \frac{Cpp'}{4}\e(g^{1/p},g^{1/p'})
\end{equation}
for all positive $g \in \H$.
\end{proof}

We now prove Theorem~\ref{monotone} using the extension of the classical Stroock-Varopoulos inequality proven in Theorem~\ref{sv}.

\begin{proof}
It is a direct consequence of the `self-dual' reformulation (\ref{self-dual}) of $p$-logSob, Theorem~\ref{sv}, and Remark~\ref{case1}.
The fact that $p$-logSob with constant $C$ implies $0$-logSob (with the same constant) for every $p \neq 0$ may be proved in two natural ways.
One can deduce the Poincar\'e inequality with constant $C/2$ from $p$-logSob by setting $f=e^{\delta g}$
and letting $\delta$ tend to zero (as in the first part of proof of Lemma~\ref{0ls=poin},
and use Lemma~\ref{0ls=poin} to finish the argument).
Alternatively, one can first deduce from $p$-logSob the $q$-logSob inequalities (with the same constant) for $q$ arbitrarily close to zero,
and then simply apply limit transition $q \to 0$.
\end{proof}

In view of Lemma~\ref{0ls=poin} and Theorem~\ref{monotone}, the $p$-logSob inequalities, $p \in [0,2]$,
can be treated as a family interpolating in a continuous and monotone way between the classical Poincar\'e and logaritmic Sobolev inequalities.
Another approach to the interpolation problem may be found in \cite{Latala00}.
Relation between the two approaches seems unclear to the present authors and perhaps it deserves some further investigation.

However, it is well known that all the $p$-logSob inequalties for
$p \in (1,2]$ are in a sense equivalent, at least if we do not care too much about constants
(we do not know whether all $p$-logSob inequalities for $p \in (0,1)$ are equivalent in a similar sense).

\begin{proposition} \label{reversinglogSob}
Let $1 < q \leq p \leq 2$. Assume that $q$-logSob holds true with a constant $C>0$. Then also $p$-logSob holds true, with constant
$\frac{(p-1)q^{2}}{(q-1)p^{2}}C$.
\end{proposition}

\begin{remark} 
It follows from  \cite[Proposition 3.1]{Bakry:94}   that  in case of  diffusions with  invariant measure $\mu$  {\em any} $q$-logSob implies {\em any} $p$-logSob with the same constant.  However, we did not find in the literature any reference to the results of the same form as Proposition~\ref{reversinglogSob} regarding reversible Markov chains. On the other hand, one can first deduce from $q$-logSob a hypercontractive inequality and then deduce $2$-logSob from it, which in turn yields $p$-logSob.
This way around was known before but it yields much worse estimates.
\end{remark}

\begin{proof}
Indeed, since $\frac{(p-1)q^{2}}{(q-1)p^{2}}=\frac{qq'}{pp'}$ it suffices to prove that
$\e(g^{1/q},g^{1/q'}) \leq \e(g^{1/p},g^{1/p'})$ for every positive
$g \in \H$, which follows easily from Lemma~\ref{comp} applied to
$I=(0,\infty)$, $\varphi_{1}(x)=x^{1/q}$, $\varphi_{2}(x)=x^{1/q'}$,
$\psi_{1}(x)=x^{1/p}$, and $\psi_{2}(x)=x^{1/p'}$. The inequality
\[
(\varphi_{1}(a)-\varphi_{1}(b))
(\varphi_{2}(a)-\varphi_{2}(b)) \leq
(\psi_{1}(a)-\psi_{1}(b))
(\psi_{2}(a)-\psi_{2}(b))
\]
is equivalent to
\[
(a/b)^{\frac{1}{p}-\frac{1}{2}}+(a/b)^{\frac{1}{2}-\frac{1}{p}}
\leq
(a/b)^{\frac{1}{q}-\frac{1}{2}}+(a/b)^{\frac{1}{2}-\frac{1}{q}}.
\]
Since for every $s>0$ the function $w \mapsto s^{w}+s^{-w}$ is non-decreasing on $[0,\infty)$, we finish the proof by setting $s=a/b$ and noting that
$\frac{1}{q}-\frac{1}{2} \geq \frac{1}{p}-\frac{1}{2} \geq 0$.
\end{proof}

Usually it is not easy to prove the classical logarithmic Sobolev inequality ($2$-logSob in our notation). On the other hand, the following lemma provides a modified logarithmic Sobolev inequality ($1$-logSob in our notation)
for a large class of simple semigroups.

\begin{lemma} (\cite{BobkovTetali06}) \label{simple}
Assume that the semigroup $(T_{t})_{t \geq 0}$ is generated by $L=Id-\E$.
Then it satisfies $1$-logSob with constant $4$, i.e., for all positive $f \in \H$ there is
\[
Ent(f) \leq \e(f, \log f).
\]
\end{lemma}

\begin{proof}
Indeed, it suffices to note that the logarithm function is concave on
$(0,\infty)$, so that $\E \log f \leq \log \E f$. Thus
\[
Ent(f)=\E f\log f - \E f \cdot \log \E f \leq
\E f\log f- \E f \cdot \E \log f=
\]
\[
\E f(\log f-\E \log f) =
\E fL\log f= \e(f,\log f).
\]
\end{proof}

\begin{remark}
Lemma~\ref{simple} was proved in \cite{BobkovTetali06}. We remark that  the constant $4$ is not always the optimal $1$-logSob constant for the semigroups generated  by $L=Id-\E$. For example,  for the  two point space $(\{ 0, 1\}, \alpha\delta_0+ (1-\alpha)\delta_1)$ the best $1$-logSob constant $C$ is known \cite{BobkovTetali06} to satisfy $C \le \frac{4}{1+ 2\sqrt{\alpha(1 - \alpha)}} <4$.  For  $\alpha = \frac 12$, the best 1-logSob constant is $2$ \cite{BobkovTetali06}. When $\alpha \ne \frac 12$, the best $1$-logSob constant  is not known (though the best $2$-logSob constant is already known \cite{DiaconisSaloff-Coste:96}). 
\end{remark}

\subsection{Tensorization} \label{sec:logSobTensorization}
 The $p$-logSob inequalities obviously share the tensorization property of the classical
logarithmic Sobolev and Poincar\'e inequalities. This is a standard observation but we include it here for reader's convenience.
For $i=1,$ $2,\ldots,$ $n$ assume that $(\Omega_{i}, \mu_{i})$ is a finite (this assumption may be relaxed)
probability space with an associated space $\H_{i}$ of real functions on $\Omega_{i}$,
and a Markov semigroup $(T^{(i)}_{t})_{t \geq 0}:\H_{i} \rightarrow \H_{i}$ generated by a self-adjoint positive
semi-definite operator $L_{i}$ (all of them enjoying properties described in the Preliminaries section).
Now let us consider a new semigroup $(T_{t})_{t \geq 0}$ of operators acting on a space $\H=\H_{1} \otimes \H_{2}
\otimes \ldots \otimes \H_{n}$
of real-valued functions on a product probability space
\[
(\Omega, \mu)=(\Omega_{1} \times \Omega_{2} \times \ldots \times \Omega_{n}, \mu_{1} \otimes \mu_{2} \otimes \ldots \otimes \mu_{n}).
\]
We obtain it by defining its generator $L:\H \rightarrow \H$ as
\[
L= \sum_{i=1}^{n}  Id_{\H_{1}} \otimes \ldots \otimes
Id_{\H_{i-1}} \otimes L_{i} \otimes Id_{\H_{i+1}} \ldots
\otimes Id_{\H_{n}}.
\]
Equivalently, we may define it by setting, for $t \geq 0$,
\[
T_{t}= T^{(1)}_{t} \otimes
T^{(2)}_{t} \otimes \ldots \otimes T^{(n)}_{t}.
\]

\begin{proposition} \label{tensor}
Let $p \in \R$. In the setting as above, assume that there exist positive constants $C_{1},$ $C_{2},\ldots,$ $C_{n}$ such that the semigroup $(T^{(i)}_{t})_{t \geq 0}$ satisfies $p$-logSob with constant $C_{i}$ for $i=1,$ $2,\ldots,$ $n$. Then the semigroup $(T_{t})_{t \geq 0}$ satisfies $p$-logSob with the constant $C=\max(C_{1}, C_{2},\ldots, C_{n})$.
\end{proposition}

\begin{proof}
We skip the proof, referring the reader to the classical tensorization argument: subadditivity of entropy (or variance, if $p=0$).
\end{proof}

\section{Hypercontractivity}

\subsection{Control of moments under semigroup action}

Let $t=t(p)$ be a differentiable nonnegative function defined on a convex subset of $\R \setminus \{ 0, 1\}$, and let $f \in \H_{(0, \infty)}$. We will study behavior of the moments of functions $f_{t(p)}:=T_{t(p)}f$. An elementary though tedious standard calculation shows that
\begin{equation} \label{der}
\frac{d}{dp} \log \| T_{t(p)}f \|_{p}=
\frac{Ent(f_{t(p)}^{p})-p^{2}t'(p)\e(f_{t(p)}^{p-1}, f_{t(p)})}{p^{2}\E f_{t(p)}^{p}}.
\end{equation}
Since, as explained in the preliminaries, $f_{t(p)}$ is also strictly positive, we may apply to it the $p$-logSob inequality, which will yield monotonicity of the map
$p \mapsto \| T_{t(p)}f\|_{p}$ upon appropriate choice
of the function $t(p)$.

\subsection{Hypercontractivity estimate}

\begin{proposition} \label{hyp}
Let $r \in (1,2]$ and let $(T_{t})_{t \geq 0}$ be a symmetric Markov semigroup.

Assume that $(T_{t})_{t \geq 0}$ satisfies $r$-logSob with constant $C$. Let
$r' \leq q \leq p$ or $1<q \leq p \leq r$. Then for every
$t \geq \frac{C}{4}\log \frac{p-1}{q-1}$ and every
$f \in \H$ there is
$\| T_{t}f\|_{p} \leq \| f\|_{q}$. In other words, $T_{t}$ is a linear contraction from $L^{q}(\Omega, \mu)$ to
$L^{p}(\Omega,\mu)$.

Conversely, if there exists $C>0$ such that
\begin{equation} \label{ass}
\| T_{\frac{C}{4}\log \frac{p-1}{q-1}}f\|_{p} \leq
\| f\|_{q}
\end{equation}
for all $p$ and $q$ such that $1<q<p \leq r$, and for all positive $f \in \H$  then $(T_{t})_{t \geq 0}$ sastisfies
$r$-logSob with the constant $C$.

\end{proposition}

\begin{remark}
That the concepts of logarithmic  Sobolev inequality  and hypercontractivity are intimately connected goes back to Gross \cite{Gross:75}. In fact, the converse part of the above proposition  follows from Theorem 1.2 of \cite{Gross:75} though we add a short proof here for the sake of completeness. But the hypothesis of the forward direction ($r$-logSob implies hypecontractivity)  of our proposition is weaker  than that of  \cite{Gross:75}  since \cite{Gross:75} assumes that $r$-logSob holds for  a nonempty open interval -  see Theorem 1.1 of \cite{Gross:75} for more details. 
 We also comment that for $r=2$ we recover the part (i) and (ii) of Theorem 3.5 of Diaconis and Saloff-Coste \cite{DiaconisSaloff-Coste:96} on the classical equivalence of the logarithmic Sobolev inequality and hypercontractivity for the reversible Markov chains (with essentially the same proof).
\end{remark}

\begin{proof}
In the proof of the first assertion without loss of generality we can assume that $f \geq 0$ - indeed, since
$T_{t}$ is order preserving, the pointwise inequality
$-|f| \leq f \leq |f|$ implies that $|T_{t}f| \leq T_{t}|f|$ pointwise, and thus $\| T_{t}f\|_{p} \leq \| T_{t}|f|\,\|_{p}$ whereas $f$ and $|f|$ have the same $q$-th norm. Furthermore, without loss of generality we may assume that
$f$ is strictly positive (which follows by considering functions $f+\varepsilon$ instead of $f$ and then letting
$\varepsilon \to 0^{+}$).

Theorem \ref{monotone} and Lemma \ref{lsdual} imply that
$(T_{t})_{t \geq 0}$ satisfies $s$-logSob with constant $C$ for all $s \in (1,r] \cup [r',\infty)$.
Let $t(s)=\frac{C}{4}\log \frac{s-1}{q-1}$, so that
$t(q)=0$. Then $s^{2}t'(s)=\frac{Cs^{2}}{4(s-1)}$ and
(\ref{der}) together with the $s$-logSob imply that the map
$s \mapsto \| T_{t(s)}f\|_{s}$ in non-increasing on $[q,p]$.
Comparing its values at the ends of the interval we arrive
at $\| T_{t_{p,q}}f\|_{p} \leq \| f\|_{q}$ for $f>0$, where
$t_{p,q}=\frac{C}{4}\log \frac{p-1}{q-1}$. To finish the proof of the first assertion for $t>t_{p,q}$ it suffices to express $T_{t}$ as $T_{t-t_{p,q}}\circ T_{t_{p,q}}$, and use the fact that the semigroup is contractive in $L^{p}$-norm ($p>1$).

To prove the second assertion let us fix some $q \in (1,r)$ and some positive $f \in \H$. For $p \in [q,r)$ let
$t(p)=\frac{C}{4}\log \frac{p-1}{q-1}$, so that $t(q)=0$.
Since the map $p \mapsto \| T_{t(p)}f\|_{p}$ is non-increasing on $[q,r)$ by using (\ref{der}) at $p=q$ we infer that $q$-logSob holds true with the constant $C$. Passing to the limit $q \to r^{-}$ ends the proof.
\end{proof}

\section{Reverse hypercontractivity - preliminary results}
In this section we prove some preliminary results regarding reverse hypercontractivity.

\subsection{Reverse contraction}

We first state the following corollary of Jensen's inequality establishing `reverse contraction':
\begin{lemma} \label{concave}
Let $I$ be a non-empty convex open subset of $\R$ and let
$(T_{t})_{t \geq 0}$ be a symmetric Markov semigroup. Then for every $t>0$ and every concave $\Phi: I \rightarrow \R$
there is $\E \Phi(T_{t}f) \geq \E \Phi(f)$ for all
$f \in \H$ with values in $I$. In particular, for every
$q <1 $ and positive $f \in \H$ we have
$\| T_{t}f\|_{q} \geq \| f\|_{q}$.
\end{lemma}

\begin{proof}
Indeed, $\Phi$ may be expressed as infimum of a family
${\mathcal{C}}_{\Phi}$ of affine functions:
\[
\Phi(x)=\inf\{ \phi(x); \phi \in {\mathcal{C}}_{\Phi} \}
\]
for $x \in I$. Thus from the pointwise inequality
$\Phi(f) \leq \phi(f)$ and positivity preserving
by $T_{t}$ we deduce
$T_{t}\Phi(f) \leq T_{t}\phi(f)=\phi(T_{t}f)$ for all
$\phi \in {\mathcal{C}}_{\Phi}$
and hence
$T_{t}\Phi(f) \leq \inf \{ \phi(T_{t}f);
\phi \in {\mathcal{C}}_{\Phi}\}=
\Phi(T_{t}f),$
pointwise, again. So
$\E \Phi(f)=\E T_{t}(\Phi(f)) \leq \E \Phi(T_{t}f)$, and we are done.
The fact that also $T_{t}f$ has values in $I$ is a consequence of the order preservation.
\end{proof}


Note that the lemma used for $I=\R$ and $\Phi(x)=-|x|^{p}$, $p \geq 1$ implies the contractivity of $(T_{t})_{t \geq 0}$
in $L^{p}$-norm.

\subsection{Duality and tensorization}

The standard statement of the duality of $L^{p}$-norms is that for $p>1$ and $f \in \H$ we have
\[
\| f\|_{p}=\sup \{ \E fg;\, \| g\|_{p'} \leq 1\}.
\]
A slightly less known observation can be found in \cite{Borell82}:
\begin{lemma} \label{rev-hold}
Let $p \in (-\infty,1)$. Then for any positive $f \in \H$
there is
\[
\| f\|_{p}=\inf \{ \E fg;\, g>0, \| g\|_{p'} \geq 1\}.
\]
\end{lemma}

We skip its proof since it is an easy exercise.

The standard duality of the $L^{p}$-norms implies that $L^{p'}(\Omega,\mu)$ is Banach space dual to $L^{p}(\Omega,\mu)$ for any $p>1$, and from the symmetry of the semigroup $(T_{t})_{t \geq 0}$ we deduce that
\[
\|T_{t}\|_{L^{p}(\Omega,\mu) \to L^{q}(\Omega,\mu)}
= \|T_{t}\|_{L^{q'}(\Omega,\mu) \to L^{p'}(\Omega,\mu)}
\]
for any $p,q>1$ and $t \geq 0$.

The case $p,q \in (-\infty,1)$ is less standard and a bit more delicate
(in particular, note that this is no longer the Banach space setting).
We will need the following auxiliary result which was previously used by Borell \cite{Borell82}.

\begin{proposition} \label{dual}
Let $p,q \in (-\infty,1)$ and $t \geq 0$. Assume that
$\| T_{t}f\|_{q} \geq \| f\|_{p}$ for every positive
$f \in \H$. Then also $\| T_{t}f\|_{p'} \geq \| f\|_{q'}$
for every positive $f \in \H$.
\end{proposition}

\begin{proof}
Indeed,
\[
\| T_{t}f\|_{p'}=
\inf\{ \E gT_{t}f;\, g>0, \| g\|_{p} \geq 1 \}=
\]
\[
\inf\{ \E fT_{t}g;\, g>0, \| g\|_{p} \geq 1 \} \geq
\inf\{ \E fh;\, h>0, \| h\|_{q} \geq 1\}=\| f\|_{q'},
\]
where we have used Lemma \ref{rev-hold}, the symmetry of
$T_{t}$, assumptions of the proposition, and again Lemma
\ref{rev-hold}.
\end{proof}

\begin{lemma}\label{lem:revhyptensorize}
Assume the set-up of Subsection \ref{sec:logSobTensorization}. Let $-\infty < q< p< 1$. If for each $1 \le i \le n$, $\|T_t^{(i)}f\|_q \ge \| f\|_p$ for all positive functions $f \in \mathcal H_i$, then $\| T_t f \|_q \ge \|f\|_p$ for all  functions $f \in \mathcal H_{(0, \infty)}$.
\end{lemma}
\begin{proof}
The proof is an easy modification of  the standard argument  for showing the usual hypercontractive inequalities tensorize where  Minkowski  inequality is to be replaced by  the reverse Minkowski inequality (Lemma~\ref{concave}). We omit details.
\end{proof}

\section{Reverse hypercontractivity - general results}

We establish an analogue of Proposition \ref{hyp} for $p$ and $q$ below $1$, extending results of Borell, \cite{Borell82}. Now we restrict our considerations to positive functions.

\begin{proposition} \label{revhyp}
Let $r \in (0,1)$ and let $(T_{t})_{t \geq 0}$ be a symmetric Markov semigroup.

Assume that $(T_{t})_{t \geq 0}$ satisfies $r$-logSob with some constant $C>0$. Let $r' \leq q \leq p \leq r$.
Then for every
$t \geq \frac{C}{4}\log \frac{1-q}{1-p}$ and every positive
$f \in \H$ there is $\| T_{t}f\|_{q} \geq \| f\|_{p}$.

Conversely, if there exists $C>0$ such that
\begin{equation} \label{assrev}
\| T_{\frac{C}{4}\log \frac{1-q}{1-p}}f\|_{q} \geq
\| f\|_{p}
\end{equation}
for all $p$ and $q$ such that $0<q<p \leq r$, and for all positive $f \in \H$  then $(T_{t})_{t \geq 0}$ satisfies
$r$-logSob with the constant $C$.
\end{proposition}
\begin{remark}
Theorem 3.3 of Bakry's lecture notes \cite{Bakry:94} established similar equivalence between reverse hypercontractivity and $r$-logSob when $r<1$. Indeed the converse part of Proposition~\ref{revhyp} follows from that. But the forward direction, which turns out to be more useful in practice, Theorem 3.3 of \cite{Bakry:94} assumes that $r$-logSob holds for all $r$ belonging to some nonempty open interval instead of a single point.
\end{remark}

\begin{proof}
Let us divide the proof of the first assertion into two basic cases: $0<q \leq p\leq r$ and $r' \leq q \leq p<0$
(and in fact we will need to prove only first of them since the second follows then by Proposition \ref{dual}). Once they are proved, the assertion for $0 \leq q \leq p\leq r$ and $r' \leq q \leq p \leq 0$ will follow by passing to a limit ($q \to 0^{+}$ and $p \to 0^{-}$, respectively), while the case $q<0<p$ will follow from
\[
\| T_{t}f\|_{q}=\| T_{t-\frac{C}{4}\log \frac{1}{1-p}}(T_{\frac{C}{4}\log \frac{1}{1-p}}f)\|_{q} \geq
\| T_{\frac{C}{4}\log \frac{1}{1-p}}f\|_{0}
\geq \| f\|_{p}	
\]
since $t-\frac{C}{4}\log \frac{1}{1-p} \geq
\frac{C}{4}\log(1-q)$ for $t \geq \frac{C}{4}\log \frac{1-q}{1-p}$ (we "glue" the two cases together at zero).

Let us assume $0<q \leq p \leq r$, then. Consider a function
$t(q)=\frac{C}{4}\log \frac{1-q}{1-p}$ defined on $(0,p]$.
Then $t(p)=0$ and $q^{2}t'(q)=\frac{Cq^{2}}{4(q-1)}$, so that by (\ref{der}) the map $q \mapsto \| T_{t(q)}f\|_{q}$ is non-increasing on $(0,p]$ because $r$-logSob implies $q$-logSob, with the same constant $C$, by Theorem~\ref{monotone}. At the right end of the interval the map takes on the value $\| f\|_{p}$, so that $\| T_{t_{p,q}}f\|_{q} \geq \| f\|_{p}$ for $t_{p,q}=\frac{C}{4}\log \frac{1-q}{1-p}$. For $t>t_{p,q}$ we simply express $T_{t}f$ as
$T_{t-t_{p,q}}(T_{t_{p,q}}f)$ and use Lemma \ref{concave}.

To prove the converse assertion, let us fix some
$p \in (0,r)$ and a positive $f \in \H$.
For $q \in (0,p]$ let
$t(q)=\frac{C}{4}\log \frac{1-q}{1-p}$, so that $t(p)=0$.
Since the map $q \mapsto \| T_{t(q)}f\|_{q}$ is non-decreasing on $(0,p]$ formula (\ref{der}) used at $q=p$ yields $p$-logSob with the constant $C$. Passing to the limit $p \to r^{-}$ ends the proof.
\end{proof}

We can now prove Theorem~\ref{thm:1ls->revhyp}. In fact we will prove the following result which includes an inverse.

\begin{corollary} \label{1ls->revhyp}
If a symmetric Markov semigroup $(T_{t})_{t \geq 0}$ satisfies $r$-logSob with constant $C$ and $r \geq 1$ then for all
$q<p<1$ and every positive $f \in \H$ for all
$t \geq \frac{C}{4}\log \frac{1-q}{1-p}$ we have
$\| T_{t}f\|_{q} \geq \| f\|_{p}$.

Conversely, if for some $C>0$ a symmetric Markov semigroup
$(T_{t})_{t \geq 0}$ satisfies
$\| T_{\frac{C}{4}\log \frac{1-q}{1-p}}f\|_{q} \geq \| f\|_{p}$ for all $0<q<p<1$ and all positive $f \in \H$ then
it also satisfies $1$-logSob with the constant $C$.
\end{corollary}

\begin{proof}
By Theorem \ref{monotone} for all $r \in (0,1)$ also $r$-logSob holds, with the same constant $C$. The assertion follows immediately from Proposition \ref{revhyp}.

The converse assertion is easy - Proposition \ref{revhyp} implies that $(T_{t})_{t \geq 0}$ satisfies $r$-logSob with the same constant $C$ for all $r \in (0,1)$, and it suffices to pass to the limit ($r \to 1^{-}$).
\end{proof}

As in~\cite{Borell82,Mossel06} we can now obtain the two function version corollary.

\begin{corollary} \label{cor:twofunction_general}
If a symmetric Markov semigroup $(T_{t})_{t \geq 0}$ satisfies $1$-logSob with constant $C$ then for all
$0< p, q <1$ and every nonnegative $f, g \in \H$ for all
$t \geq -\frac{C}{4}\log[ (1-p)(1-q)]$ we have
$ \E [f T_t g ]  \geq \| f\|_{p} \| g\|_q$.
\end{corollary}

\begin{proof}
Fix $0 <p, q<1$ and $t \geq \frac{C}{4}\log[ (1-p)(1-q)]$. Approximating the nonnegative functions $f$ and $g$ by positive functions $f+ \epsilon$ and $g+\epsilon$ and then in the end letting $\epsilon \downarrow 0$,  we can assume, without loss of generality, that  the functions $f, g \in \H$ are positive.
Applying the reverse H\"{o}lder's inequality (Lemma \ref{rev-hold}), we have $\E [f T_t g] \ge \|f \|_p \|T_t g \|_{p'}$. It remains to show that $ \| T_t g \|_{p'}   \ge \| g \|_{q}$,
which immediately follows from Corollary \ref{1ls->revhyp} once we note that $ (1- p') = (1-p)^{-1}$.
\end{proof}

We conclude this section by proving Corollary~\ref{simrevhyp}.

\begin{proof}
Indeed, Lemma \ref{simple} and Proposition \ref{tensor}
(in the product case) imply that $(T_{t})_{t \geq 0}$ satisfies $1$-logSob with constant $4$, so that it suffices to use Corollary \ref{1ls->revhyp}.
\end{proof}

\section{Improved reverse bounds for simple semigroups}

Actually, we can significantly weaken the condition
$t \geq \log \frac{1-q}{1-p}$ in Corollary \ref{simrevhyp} for simple operators and prove Theorem~\ref{simrevhyp_strong}.


\begin{proof}
It suffices to prove the claim in the case $q<p \leq 0$ - Proposition \ref{dual} together with an observation that for $0 \leq q<p<1$ there is $p'<q' \leq 0$ and
\[
\frac{2-p'}{2-q'}=\frac{(1-q)(2-p)}{(1-p)(2-q)}
\]
will do the rest. Also, we can restrict to the case $L=Id-\E$
- the product case will follow by Lemma~\ref{lem:revhyptensorize}.

Let $q<p<0$ and $L=Id-\E$, then, so that
\[
T_{t}f=e^{-tL}f=\E f +e^{-t}(f-\E f)=
\E f+\tilde{\theta}(f- \E f),
\]
where $\tilde{\theta}:=e^{-t} \leq \theta:=\frac{2-p}{2-q} \in (0,1)$. For $s<0$, define
$\Psi_{s}: (-1,\infty) \rightarrow [0,\infty)$,
\[
\Psi_{s}(x)=\frac{1+sx-(1+x)^{s}}{s}.
\]
It is easy to check that $\Psi_{s}$ is a convex function
with $\Psi_{s}(0)=\Psi_{s}'(0)=0$ and
$\Psi_{s}''(x)=(1-s)(1+x)^{s-2}$ (actually, the same properties hold true in the case $s>0$, and in the case $s=0$
with $\Psi_{0}(x)=x-\log(1+x)$, but we will not need those).
The inequality
\begin{equation} \label{psiineq}
\Psi_{q}(\theta x) \leq \Psi_{p}(x)
\end{equation}
holds true for every $x \in (-1,\infty)$. Indeed, due to the
properties of $\Psi_{s}$ listed above it suffices to prove that
\begin{equation} \label{intermed}
\theta^{2}\Psi_{q}''(\theta x) \leq \Psi_{p}''(x).
\end{equation}
This is equivalent to the inequality
\begin{equation}\label{ineq:key_simple_semigr}
\frac{(1+x)^{\theta}}{(1+\theta x)} \leq
\Big(\frac{(2-q)^{2}(1-p)}{(2-p)^{2}(1-q)}\Big)^{\frac{1}
{2-q}}
\end{equation}
which immediately follows from the elementary inequality
$(1+x)^{\theta} \leq 1+\theta x$, and from the fact that
the map
\[
s \mapsto (2-s)^{-2}(1-s)=\frac{1}{2-s}-
\Big(\frac{1}{2-s}\Big)^{2}
\]
is positive and non-decreasing on $(-\infty,0)$.
We are to prove $\| T_{t}f\|_{q} \geq \| f\|_{p}$ for every positive $f \in \H$. By the homogeneity, we may and will assume that $\E f=1$, so that $g=f-1$ and
$\tilde{g}=\tilde{\theta}\theta^{-1}g$ are zero-mean and take values in $(-1, \infty)$. Then we have
\[
\| T_{t}f\|_{q}^{p}=\| 1+\tilde{\theta}g\|_{q}^{p}=
\| 1+\theta \tilde{g}\|_{q}^{p}=
(\E (1+\theta \tilde{g})^{q})^{p/q}=
(1-q\E \Psi_{q}(\theta \tilde{g}))^{|p|/|q|} \leq
\]
\[
1-\frac{|p|}{|q|}q\E \Psi_{q}(\theta \tilde{g})=
1+|p|\E \Psi_{q}(\theta \tilde{g})
\stackrel{(\ref{psiineq})}{\leq}
1+|p|\E \Psi_{p}(\tilde{g})=
1+|p|\E \Psi_{p}(\tilde{\theta}\theta^{-1}g)=
\]
\[
1+|p|\E \Psi_{p}\Big(\tilde{\theta}\theta^{-1}g+
(1-\tilde{\theta}\theta^{-1})\cdot 0\Big) \leq
1+|p|\E \Big(\tilde{\theta}\theta^{-1}\Psi_{p}(g)+
(1-\tilde{\theta}\theta^{-1})\Psi_{p}(0)\Big)
\]
\[
\leq
1+|p|\E \Psi_{p}(g)=
1-p\E \Psi_{p}(g)=
\E(1+g)^{p}=\E f^{p}=\| f\|_{p}^{p}
\]
which ends the proof - recall that the exponent $p$ is negative. The first inequality above was just an application
of the elementary $(1+x)^{a} \leq 1+ax$, with $a=|p|/|q| \in (0,1)$, $x>-1$. The last inequality follows from the fact that $\Psi_s$ obtains its minimum at $s=0$. 

The case $p=0$ follows by an obvious limit transition.
\end{proof}

\begin{remark} \label{rem:zero_case}
Note that we could obtain a better reverse hypercontractivity constant than those given in Theorem~\ref{simrevhyp_strong} by first maximizing  the function
$\Psi_{q}''(\theta x)/ \Psi_{p}''(x)$
over $x \in (-1, \infty)$ and then by trying to solve for $\theta$ in terms of $p$ and $q$ such that
\eqref{intermed}
holds. This would lead to an equation of the form
\begin{equation*} \label{gamma_relation}
\left( \frac{1- \theta }{1-r} \right)^{p-q} = \frac{1-q}{1-p} \theta^p r^{2-p} \quad \text{ where } r = (2-p)/(2-q).
\end{equation*}
But unfortunately, in general,  $\theta$ can not be recovered explicitly from the above equation.
 However, in the special case when $-\infty < q < p=0$, $\theta$ can explicitly be solved as 
\[ \theta(q) =  1 + q (2-q)^{-1+2/q} [4(1-q)]^{-1/q}=\]
\[
1+\frac{1}{2}q+\frac{3}{8}q^{2}+O(q^{3})=1-\frac{1}{2}q'-\frac{1}{8}(q')^{2}+O((q')^{3})
\]
as $q \rightarrow 0^{-}$ and, equivalently, $q' \rightarrow 0^{+}$.
Thus, under assumptions of Theorem~\ref{simrevhyp_strong} about the semigroup,
there exists a function $\eta: (-\infty,0) \longrightarrow (0,\infty)$ given by $\eta(q)=-\log \theta(q)$, with $\eta(q)=
-\frac{1}{2}q-\frac{1}{4}q^{2}+O(q^{3})$ as $q \rightarrow 0^{-}$, such that for all $q<0$ and positive $f$ we have
$
\| T_{t}f\|_{q} \geq \| f\|_{0}
$
for every $t \geq \eta(q)$.
Also, by duality, there exists a function $\tau: (0,1) \longrightarrow (0,\infty)$ given by $\tau(p)=-\log \theta(p')$,
with $\tau(p)=\frac{1}{2}p+\frac{1}{4}p^{2}+O(p^{3})$ as $p \rightarrow 0^{+}$, such that for all
$p \in (0,1)$ and positive $f$ we have
$
\| T_{t}f\|_{0} \geq \| f\|_{p} 
$
for every $t \geq \tau(p)$.
\end{remark}

\begin{corollary}\label{cor:two_function}
If a symmetric Markov semigroup $(T_{t})_{t \geq 0}$ has a simple generator $L=Id-\E$ or it is a tensor product of such simple semigroups, then for all $0 <p, q<1$ and all nonnegative
$f, g \in \H$, we have
\begin{equation} \label{ineq:two_function}
\E [f T_t g] \ge \|f \|_p \|g \|_q,
\end{equation}
for all $t \geq \log \frac{(2-p)(2-q)}{4(1-p)(1-q)}$.
\end{corollary}
\begin{proof} Fix $0 <p, q<1$ and $t \geq \log \frac{(2-p)(2-q)}{4(1-p)(1-q)}$. Approximating the nonnegative functions by postive functions if necessary, we can assume, without loss of generality, that  the functions $f, g \in \H$ are positive.
Applying the reverse H\"{o}lder's inequality (Lemma \ref{rev-hold}), we have $\E [f T_t g] \ge \|f \|_p \|T_t g \|_{p'}$.
Now consider $t_1, t_2 >0 $ such that $t_1 = \log \frac{2-p'}{2}  = \log \frac{(2-p)}{2(1-p)}  $ and $t_2 =  \log \frac{(2-q)}{2(1-q)}$. Note that $t \ge t_1 +t_2$. Thus, using the semigroup property, we can write  $T_t  = T_{t_1} \circ T_{t_2} \circ T_{t - (t_1+t_2)}$. Therefore we conclude that
\[ \| T_t g \|_{p'} \ge \|T_{t_2} \circ T_{t - (t_1+t_2)} g \|_{0}  \ge  \| T_{t - (t_1+t_2)} g \|_{q}   \ge \| g \|_{q},  \]
where  we used Theorem~\ref{simrevhyp_strong} in the first and the second inequality and Lemma \ref{concave} in the third inequality.
\end{proof}

We now obtain the following corollary regarding {\em $\rho$-correlation}.
\begin{definition} \label{d:rho_corr}
Consider a product space
$(\Omega, \mu) = (\prod_{i=1}^n \Omega_i, \otimes_{i=1}^n \mu_i)$ where $(\Omega_i, \mu_i)$ are finite probability spaces. We say that $(x,y) \in \Omega^2$ are {\em $\rho$-correlated} if $x$ is distributed according to $\mu$ and  the conditional distribution of $y$ given $x$ is given as follows:
for each $i$ independently, with probability $\rho$, $y_i = x_i$ and with probability $1-\rho$, $y_i$ is sampled independently from
$\mu_i$.
\end{definition}

\begin{lemma}\label{l:correlated_sets_noise}
Let $(\Omega, \mu)$ be the product probability space in Definition~\ref{d:rho_corr}. Let $A, B \subseteq \Omega$ be two sets such that
$\mu\{A\}, \mu\{B\} \ge \eps \ge 0$. 
Let $x$ be distributed according to the product measure $\mu$ and $y$ be a $\rho$-correlated copy of $x$ for some $ 0 \le \rho < 1$. Then
\begin{equation} \label{ineq:corr}
\PP\{ x \in A, y \in B\} \ge \eps^{\frac{2- \sqrt{\rho}}{1 - \sqrt{\rho}}}.
\end{equation}
\end{lemma}

\begin{proof}
Let $f$ and $g$ be the characteristic functions of the sets $A$ and $B$ respectively. Note that
\[ \PP\{ x \in A, y \in B\} = \E [f(x) g(y)]  = \E \big[ f(x) \E [g(y)| x] \big]=   \E[ f T_t g ],\]
where $t=\log(1/\rho)$ and $T_t = \otimes_{i=1}^n T^i_t$ where $T^i_t = e^{-t(Id-\E)}$.
So, by Corollary~\ref{cor:two_function}, we have
\begin{equation} \E[ f T_t g ] \ge \|f \|_p \|g\|_q, \label{eq:iso_intermediate}
\end{equation}
for all $0 < p, q<1$ such that $ \rho = \frac{4(1- p)(1-q)}{(2-p)(2-q)}$. We now take $p = q = \frac{2(1-\sqrt{\rho})}{2 - \sqrt{\rho}}$ in \eqref{eq:iso_intermediate} to conclude the proof.
\end{proof}

\begin{remark}
We can also use Corollary~\ref{cor:twofunction_general} which deals with general symmetric Markov semigroups  to get a lower bound $\eps^{  \frac{2}{1 - \sqrt{\rho}}}$. But this bound is worse than what we have achieved by using Corollary~\ref{cor:two_function} that improves on the bounds provided by Corollary~\ref{cor:twofunction_general} in the case of simple semigroups.
\end{remark}

\begin{remark} \label{r:corr_optimality}
The lower bound in Lemma \ref{l:correlated_sets_noise} does not depend on the marginal measures $\mu_i, 1\le i \le n$. Obviously, it is
cannot be improved for $\rho=0$. However, it is also quite close to optimal for $\rho$ close to $1$. To see it, consider $\Omega_{i}=\{ -1,1\}$ and
$\mu_{i}=\frac{1}{2}\delta_{-1}+\frac{1}{2}\delta_{1}$ for all $i$'s. Let $c=c(\eps)$ be such that
$(2\pi)^{-1/2}\int_{c}^{\infty} e^{-u^{2}/2}\,du=\eps$, and let $A=\{ z: n^{-1/2}\sum_{i=1}^{n} z_{i} \leq -c \}$
and $B=\{ z: n^{-1/2}\sum_{i=1}^{n} z_{i} \geq c \}$. Finally, let $(G_{1}, G_{2})$ be a Gaussian random vector
with ${\mathcal{N}}(0,1)$ marginals, such that $\E [G_{1}G_{2}]=\rho$.
By the Central Limit Theorem $\mu\{A\}$ and $\mu\{B\}$ tend to $\eps$ as $n \rightarrow \infty$, while
$\PP\{ x \in A, y \in B\}$ tends to $\PP\{G_{1} \leq -c, G_{2} \geq c\}$.
Now it suffices to note that
\[
\log_{\eps}\PP\{G_{1} \leq -c, G_{2} \geq c\} \geq \log_{\eps}\PP\{ G_{2}-G_{1} \geq 2c \}
\stackrel{\eps \to 0^{+}}{\longrightarrow} \frac{2}{1-\rho}
\]
and
\[
\frac{2-\sqrt{\rho}}{1-\sqrt{\rho}}-\frac{1}{2} \leq
\frac{2}{1-\rho}
\leq \frac{2-\sqrt{\rho}}{1-\sqrt{\rho}},
\]
which holds for all $\rho \in [0,1)$. We skip some tedious but straightforward calculations.
\end{remark}

\begin{remark} \label{rem:lastmoment} Under assumptions of Lemma~\ref{l:correlated_sets_noise} we also have
\begin{equation} \label{ineq:lastmoment}
\PP\{ x \in A, y \in B\} \ge \eps^{\frac{2}{1-\rho}+\kappa \cdot (1-\rho)},
\end{equation}
where $\kappa$ is some universal constant. This is a significant strengthening when $\rho$ is close to $1$, especially in view of
Remark~\ref{r:corr_optimality}.

Indeed, it suffices to notice that in the proof of Corollary~\ref{cor:two_function} one can take $t_{1}=\tau(p)$
and $t_{2}=\tau(q)$, using Remark~\ref{rem:zero_case} rather than Theorem~\ref{simrevhyp_strong}. Thus
(\ref{ineq:two_function})
holds true for all $t \geq \tau(p)+\tau(q)$. Let us set $p=q=1-\rho-C \cdot (1-\rho)^{3}$.
The asymptotic behavior of $\tau(p)$ established in Remark~\ref{rem:zero_case} implies that $e^{-2\tau(p)}=1-p+O(p^{3})$
as $p \rightarrow 0^{+}$. Thus, by choosing the constant $C$ large enough and $\hat{\rho} \in (0,1)$ close enough
to $1$, we prove that for $\rho \in (\hat{\rho},1)$ there is $e^{-2\tau(p)} \geq \rho$, i.e.,
$2\tau(p) \leq \log(1/\rho)$, and therefore (\ref{ineq:two_function}) holds true for $t=\log(1/\rho)$.
By repeating the proof of Lemma~\ref{l:correlated_sets_noise} we arrive at 
\[
\PP\{ x \in A, y \in B\} \ge \eps^{2/p}=\eps^{\frac{2}{1-\rho-C \cdot (1-\rho)^{3}}}
\]
for $\rho \in (\hat{\rho},1)$. This, together with (\ref{ineq:corr}) used for $\rho \leq \hat{\rho}$,
yields (\ref{ineq:lastmoment}).

A similar asymptotic strengthening applies to many further results of the next two sections (whenever one deals with simple semigroups and their tensor products, and also in Section~\ref{sec:non-simple} for $\alpha, \gas$ close to zero) but we will omit these generalizations for the sake of brevity.
\end{remark}

\section{Reverse hypercontractivity for some non-simple operators} \label{sec:non-simple}
For some of the applications afterwards we will be interested in operators that are not necessarily simple but are obtained by composing a simple operator
with a non-simple operator. In this section we extend some of the reverse hypercontractive results to this setup.

\begin{proposition}\label{simrevhyp_strong_general}
Assume that $(\Omega, \mu)$ is a finite probability space and $K$ is Markov kernel on $\Omega$.
Let  $\nu = \mu K$ and
\begin{equation}
\alpha := \min_{ x, y: \nu\{y\} > 0} \frac{K(x, y)} { \nu\{y\}} > 0.  \label{min_atom}
\end{equation}
Let $\gas = -\log(1-\alpha)$.  Let the operator $K^{\otimes n}$ be the $n$-fold tensor product
of the kernel $K$  on the product space $(\Omega^n, \mu^{\otimes n})$.
Then for all $f : \Omega^n \to \mathbb R_+$, for all $q<p \leq 0$ and $\gas \geq \log \frac{2-q}{2-p}$, and also for all $0 \leq q<p<1$ and
$\gas \geq \log \frac{(1-q)(2-p)}{(1-p)(2-q)}$
we have
$\| K^{\otimes n} f\|_{L^q(\mu^{\otimes n})} \geq \| f\|_{L^p(\nu^{\otimes n})}$.
\end{proposition}

\begin{proof}
Some readers may find it convenient and natural to interpret the abstract operations of this proof in terms of
matrix multiplication - then measures and functions should be represented, respectively, by horizontal and
vertical vectors with coordinates indexed by elements of $\Omega$, and Markov operators become stochastic
matrices (for example $\E_{\mu}$ should be understood as a matrix with all rows equal to the vector representing
measure $\mu$). First assume that $n=1$. Let $T_t = e^{-t} I + (1-e^{-t}) \E_\mu,  t \ge 0 $ be the simple Markov semigroup on  $(\Omega, \mu)$.
We can extend the definition of $T_t$ for $t <0$ as well,
though it no longer is a Markov operator unless $|\Omega|=1$.
However, it is easy to check that $T_{-t} K = e^t K - (e^t-1) \E_\nu$ is a Markov operator for small enough $t >0$.
First, for all $t$, $(T_{-t}K)1=T_{-t}(K1)=T_{-t}1=1$, so it remains to check the positivity of $T_{-t} K$.
Second, to ensure positivity, we need to show that $ e^t K(x, y)  - (e^t - 1) \nu(y) \ge 0$
for all $x, y \in \Omega$ which holds if $0 \le  t \le -\log(1-\alpha) = \gas$.

Thus $S := T_{-\gas} \circ K$ is Markovian and the kernel $K$ can be written as composition of two Markov
kernels in the  following way:
\begin{equation} \label{decomp}
 K = T_{\gas}S.
 \end{equation}
For every probability measure $\rho$ on $\Omega$ we have $\rho \E_{\mu}=\mu$, in particular for $\rho=\mu$, and
therefore $\mu T_{t}=\mu$ for every real $t$. Hence \eqref{decomp} implies
\[
\mu S = (\mu T_{\gas})S  = \mu (T_{\gas}S) = \mu K  = \nu.
\]

Using the decomposition~\eqref {decomp}, by Theorem \ref{simrevhyp_strong} we obtain
\[
\| K f\|_{L^q(\mu)} = \| T_{\gas}(Sf)\|_{L^{q}(\mu)}
\ge  \| S f\|_{L^p(\mu)} \geq  \| f\|_{L^p(\nu)},
\]
for $p$ and $q$ as  in the hypothesis.
Since $\E_{\nu}[f^{p}]=(\mu S)f^{p}=\E_{\mu}[S(f^{p})]$,
the last inequality follows from the pointwise (coordinatewise) inequalities
$(Sf)^{p} \geq S(f^{p})$ for $p \in (0,1)$ and $(Sf)^{p} \leq S(f^{p})$ for $p<0$.
They, in turn, follow from the Markovianity of $S$ and concavity (resp. convexity) of the function
$(0,\infty) \ni t \mapsto t^{p}$ for $p \in (0,1)$ (resp. $p<0$).
The proof for general $n$ now follows from standard tensorization argument.
\end{proof}

\begin{corollary}\label{cor:two_function_general}
Consider the set-up of Proposition~\ref{simrevhyp_strong_general}. Then for all $0 <p, q<1$ and all nonnegative
$f, g$, we have
\[ \E [f K^{\otimes n} g] \ge \|f \|_{L^q(\mu^{\otimes n})} \|g \|_{L^p(\nu^{\otimes n})}, \]
for all $\gas \geq \log \frac{(2-p)(2-q)}{4(1-p)(1-q)}$.
\end{corollary}
\begin{proof}
Same as Corollary~\ref{cor:two_function}.
\end{proof}

\begin{lemma} \label{l:corr_sets_noise_cube}
Let $(x_i, y_i)_{1 \le i \le n} $ be i.i.d. $\Omega^2$-valued random variables.  Let $\mu$ and $\nu$ be
the marginal distributions of $x_i$ and $y_i$ respectively and let $K$ denote the conditional probability kernel
$K(a, b) = \PP\{ y_i =b| x_i = a\}$. Assume that $\alpha >0$ where $\alpha$ is given in \eqref{min_atom}, so that
$\PP\{x_{i}=a, y_{i}=b\} \geq \alpha \mu\{a\}\nu\{b\}$ for all $a,b \in \Omega$ and all $1 \leq i \leq n$.
Then for any two sets $A, B \subseteq \Omega^n$ such that $ \mu^{\otimes n}\{ A\}, \nu^{\otimes n}\{ B\} \geq \eps \geq 0$ we have
\begin{equation} \label{eq:two_small}
\PP \{ x \in A, y \in B \} \geq  \eps^{\frac{2- \sqrt{1-\alpha}}{1 - \sqrt{1-\alpha}}},
\end{equation}
where $x = (x_1, x_2, \ldots, x_n)$ and $y = (y_1, y_2, \ldots, y_n)$ are $\Omega^n$-valued random variables.
\end{lemma}

\begin{proof}
Let $f$ and $g$ be the characteristic function of the sets $A$ and $B$ respectively. Note that
\[ \PP\{ x \in A, y \in B\} = \E [f(x) g(y)]  = \E_{\mu^{\otimes n}} \big[ f(x) \E [g(y)| x] \big]=   \E_{\mu^{\otimes n}} [ f K^{\otimes n} g ].\]
Now by Corollary~\ref{cor:two_function_general},
\[ \E_{\mu^{\otimes n}}  [f K^{\otimes n} g] \ge \|f \|_{L^p(\mu^{\otimes n})} \|g \|_{L^q(\nu^{\otimes n})}, \]
for all $0 < p, q<1$ such that $1-\alpha = \frac{4(1-p)(1-q)}{(2-p)(2-q)}$.  We  take $p = q = \frac{2(1-\sqrt{1-\alpha})}{2 - \sqrt{1-\alpha}}$ to conclude the proof.
\end{proof}

\begin{remark}
The following example shows that the condition  $\alpha>0$ cannot be dropped in general. Take $\Omega = \{ 0, 1\}$ and $\mu$ to be the unbiased Bernoulli measure on $\Omega$. Let the kernel $K$ be as follows:
\[ K  = \begin{pmatrix} 0 &1\\ 1/2 & 1/2 \end{pmatrix}\]
so that $\nu=(1/4)\delta_{0} + (3/4) \delta_{1}$. Now take $A = \{x_1 = 0 \}$ and $B = \{ y_1 =0\}$. Then $\mu^{\otimes n}\{ A\} = 1/2$ and $\nu^{\otimes n}\{ B\}  = 1/4$ but $\PP \{ x \in A, y \in B \}=0$.
\end{remark}

\section{Mixing of Markov chains for big sets}\label{sec:mix}

In this section we prove Theorem~\ref{t:intersection} which establishes mixing for Markov chains satisfying $1$-logSob.
We then give a number of examples where the theorem can be applied to yield new results on mixing of Markov chains starting from big sets.
We begin with a proof of the theorem:

\begin{proof}[Proof of Theorem~\ref{t:intersection}]
Take $f$ and $g$ to be the characteristic functions of $A$
and $B$, respectively.  Then by
Corollary~\ref{cor:twofunction_general}, for any choice of $0< p,
q < 1$ with $(1-p)(1-q) = e^{-4t/C}$, we get
\begin{equation}
\PP\{ X_0 \in A, X_t \in B \} = \E [f T_t g]  \geq \|f\|_p \|g\|_q = \exp(-a^2/2p) \exp(-b^2/2q).
\label{eqn:iso}
\end{equation}
By setting $p=\frac{1- e^{-4t/C} }{1+ e^{-2t/C} (b/a)}$ and $q=\frac{1- e^{-4t/C} }{1+ e^{-2t/C} (a/b)}$ (this choice follows from a simple optimization) we conclude the proof.
\end{proof}


\subsection{Short walks on general product spaces}~\label{subsec:short}
The first example is simply a random walk on a product space.  Let $(\Omega, \mu) $ be a finite probability space and $n \ge 1$. Consider the hypercube $(\Omega^n, \mu^{\otimes n})$.
The {\em continuous-time random walk} on this space corresponds to selecting one coordinate uniformly at random with each ring of a Poisson clock (with intensity $1$) and updating that coordinate according to the distribution $\mu$. This is a reversible Markov chain with invariant distribution $\mu^{\otimes n}$. When $\Omega = \{0, 1\}$ and $\mu\{0\} = \mu\{1\} = 1/2$, we have the standard continuous-time random walk on the hypercube.
Let $L = Id - \E$ be the generator of the simple Markov semigroup on $(\Omega, \mu) $ and $(T_t)_{t \ge 0}$ be the corresponding semigroup, then the generator of the random walk on the general hypercube is given by
\[ L^{\mathrm{prod}}=  \frac{1}{n} \sum_{ i=1}^n Id \otimes  Id \otimes \cdots  \otimes \underbrace{L}_{i}  \otimes \cdots \otimes Id, \]
and the corresponding Markov semigroup can be expressed as
\[ T^\mathrm{prod}_t = T_{t/n} \otimes T_{t/n} \otimes \cdots \otimes T_{t/n}, \quad \text{for } t \ge 0.\]
From Lemma~\ref{simple}, Proposition~\ref{tensor} and Theorem~\ref{t:intersection}
it follows that:
\begin{corollary}
Let $X_t$ be the continuous-time random walk on the general hypercube $(\Omega^n, \mu^{\otimes n})$ with $X_0$
distributed according to the product measure $\mu^{\otimes n}$. Let $a, b \geq 0$ and $\tau>0$.
Then for any $A, B \subseteq\Omega^n$ with $\mu^{\otimes n} \{A\} = e^{ -a^2/2}$ and  $\mu^{\otimes n} \{B\} = e^{ -b^2/2}$, and for $t \ge \tau n$, we have
\begin{equation*}
\PP\{ X_0 \in A, X_t \in B \} \ge \exp \left( - \frac{1}{2} \frac{a^2 + 2 e^{-\tau/2}ab + b^2 }{1- e^{-\tau}} \right).
\end{equation*}
\end{corollary}

In fact, a much better bound can be obtained by repeating the proof of Theorem~\ref{t:intersection} with the two function bound in Corollary~\ref{cor:twofunction_general} which applies to simple operators and their tensors:

\begin{proposition} \label{p:ct_rw_improved}
Let $X_t$ be the continuous-time random walk on the general hypercube $(\Omega^n, \mu^{\otimes n})$ with $X_0$
distributed according to the product measure $\mu^{\otimes n}$. Let $a, b \geq 0$ and $\tau>0$.
Then for any $A, B \subseteq\Omega^n$ with $\mu^{\otimes n} \{A\} = e^{ -a^2/2}$ and  $\mu^{\otimes n} \{B\} = e^{ -b^2/2}$, and for $t \ge \tau n$, we have
\begin{equation*}
\PP\{ X_0 \in A, X_t \in B \} \ge \exp \left( -\frac{(2-e^{-\tau})(a^2 + b^2)+ 2 e^{-\tau/2}ab}{4(1- e^{-\tau})} \right) \geq
e^{-\frac{a^{2}+b^{2}-ab}{4}}\exp\left(-\frac{(a+b)^{2}}{4(1-e^{-\tau})}\right).
\end{equation*}
\end{proposition}

\begin{proof}
Note that the pair $(X_{0}, X_{t})$ is $\rho$-correlated in the sense of Definition~\ref{d:rho_corr}, with $\rho=e^{-t/n}$.
Let $p=\frac{(2-2\rho)a}{b\sqrt{\rho}+(2-\rho)a}$ and $q=\frac{(2-2\rho)b}{a\sqrt{\rho}+(2-\rho)b}$, so that $p,q \in (0,1)$ and
$\frac{4(1-p)(1-q)}{(2-p)(2-q)}=\rho$. By Corollary~\ref{cor:two_function} applied to $f=1_{A}$ and $g=1_{B}$ we have
\[
\PP\{ X_0 \in A, X_t \in B \}=\E[fT^\mathrm{prod}_t g] \geq \| f\|_{p}\| g\|_{q}=\exp\left(-\frac{a^{2}}{2p}-\frac{b^{2}}{2q}\right)
\]
and the first inequality of the assertion follows easily. The second inequality in the assertion of the proposition is elementary.
\end{proof}

\begin{remark}
The bound of Proposition~\ref{p:ct_rw_improved} is quite tight, especially for small values of $\tau$, $\mu^{\otimes n} \{A\}$, and
$\mu^{\otimes n} \{B\}$. Indeed, let us fix $t=\tau n$, choose $\alpha=\alpha(a), \beta=\beta(b)$ such that
$(2\pi)^{-1/2}\int_{\alpha}^{\infty} e^{-u^{2}/2}\,du=e^{-a^{2}/2}$ and $(2\pi)^{-1/2}\int_{\beta}^{\infty} e^{-u^{2}/2}\,du=e^{-b^{2}/2}$, and then define two subsets of the discrete cube, $A=\{ z: n^{-1/2}\sum_{i=1}^{n} z_{i} \leq -\alpha \}$ and
$B=\{ z: n^{-1/2}\sum_{i=1}^{n} z_{i} \geq \beta \}$. Now it suffices to use the CLT as in Remark~\ref{r:corr_optimality} 
(recall that in our setting $\rho=e^{-\tau}$) and observe that
\[
(1-e^{-\tau})\log \PP\{ G_{1} \leq -\alpha, G_{2} \geq \beta \} \leq (1-e^{-\tau})\log \PP\{ G_{2}-G_{1} \geq \alpha+\beta\}
\stackrel{\tau \to 0^{+}}{\longrightarrow} -(\alpha+\beta)^{2}/4
\]
while $\lim_{a \to \infty} \alpha(a)/a=1$ and $\lim_{b \to \infty} \beta(b)/b=1$. We skip tedious but quite standard calculations.
\end{remark}

We note that the mixing time of the above walk is of order $n \log n$. Therefore using the mixing time it is impossible to obtain effective bounds even when one of the sets $A$ or $B$ has a large measure and
$t$ is of order $n$. 

\subsection{An example from queueing theory} \label{subsec:queue} 
In this subsection, we will give an example where we will show reverse hypercontractivity for Markov semigroup arising from a standard q/q/$\infty$ process (defined below).  We will not use  any knowledge about the $p$-logSob constants of the semigroup but establish reverse hypercontractivity by taking Poissonian limit of the reverse hypercontractive estimate for $n$-dimensional hypercube with product Bernoulli measure ($p = \lambda/n$).  The example is of interest for a number of reasons:
\begin{itemize}
\item
It deals with a Markov chain defined on an infinite state space.
\item
It is an example where the $2$-logSob and the mixing time are both infinite (see~\cite{BobLedoux98}, the fact that the mixing time is infinite is trivial), yet it is possible to obtain
reverse hypercontractive and mixing estimates.
\item
It is a natural example for queueing theory.
\end{itemize}
Let $\mu_p$ be the Bernoulli measure $  (1-p)\delta_0 + p \delta_1$ on $\{0, 1\}$. Let $(X^{(n)}_t)_{ t \ge 0}$ be the Markov process on state space $\{0, 1\}$ corresponding to the simple semigroup generated by  $I - \E$ w.r.t.\ the measure $\mu_p$ with $p = \lambda/n, \lambda>0$ fixed. Let $ X^{n, 1}, X^{n, 2}, \ldots,$ $ X^{n, n}$ be i.i.d.\ copies of $X^{(n)}$. The process $Y^{(n)}_t := X^{n, 1}_t+ X^{n, 2}_t + \ldots+ X^{n, n}_t$  is again Markov (with state space $\mathbb N$) whose generator $L^{(n)}$ satisfies
\[ L^{(n)} f(x_1+ x_2 + \ldots+ x_n) = (I - \E_{\mu_{\lambda/n}})^{\otimes n} \hat f(x_1, x_2, \ldots, x_n), \quad x_i \in \{0, 1\}, \]
where $\hat f : \{0,1\}^n \to \mathbb R$ is given by the relation $\hat f(x_1, x_2, \ldots, x_n)  = f(x_1+x_2 +\ldots+x_n)$. Clearly, $\nu_n = \mu_{\lambda/n}^{*n }$, the $n$-fold convolution of $\mu_{\lambda/n}$, is the reversible measure of $Y^{(n)}$.
A simple calculation yields
\[  L^{(n)} f(k)  =  \big (1- \tfrac{\lambda}{n} \big )k \big (f(k) - f(k-1) \big ) + \tfrac{\lambda}{n}(n-k) \big(f(k) - f(k+1) \big).    \]
So, as $n \to \infty$,  the sequence of generators $L^{(n)}$ converges to the generator $L$ which is given by
\[ L f(k) = - (k+ \lambda) f(k) - kf(k-1) - \lambda f(k+1),  \]
for all $f :\mathbb N \to \mathbb R$. One can easily recognize the above generator as the generator for the well-known q/q/$\infty$ process which we denote by $(Y_t)_{ t\ge 0}$.  Thus $(Y_t)_{ t\ge 0}$ is a continuous time Markov process taking values in non-negative integers where $Y_t $ represents the number of customers in the  queue at time $t$ in the following set-up.  There are infinite number of servers, the customers arrive according to a Poisson process with rate $\lambda$ and the service time of each customer follows an independent exponential with mean $1$. This process is reversible w.r.t.\ $\lim_n \nu_n  = \mathrm{Poisson}(\lambda)$.
Since convergence of the generator implies the convergence of the process, we have,
for each $t \ge 0$,
\begin{equation}\label{conv_poi}
 Y^{(n)}_t \stackrel{d}{\to} Y_t,  \quad \text{ when } Y^{(n)}_ 0  = Y_0.
 \end{equation}
Here $\stackrel{d}{\to}$ means the convergence in distribution.

Let $(T^{(n)}_t)_{t\ge 0}$ (resp.\ $(T_t)_{t \ge 0}$)  be the semigroup corresponding to the Markov process $(Y^{(n)}_t)_{t \ge 0}$  (resp. $(Y_t)_{ t \ge 0}$).

 By Theorem~\ref{simrevhyp_strong}, for  any bounded $f :\mathbb N \to \mathbb R_+$,
 $0 \leq q< p< 1$ and $n \ge 1$ ,
\[
\|T^{(n)}_t f\|_{L^q(\nu_n)} \ge \|f \|_{L^p(\nu_n)}, \quad \text{for } t \ge \log \frac{(1-q)(2-p)}{(1-p)(2-q)}.
\]
Letting $n \to \infty$, by \eqref{conv_poi}, we conclude that
\[  \|T_t f\|_{L^q(\mathrm{Poi}(\lambda))} \ge \|f \|_{L^p(\mathrm{Poi}(\lambda))}, \quad \text{for } t \ge \log \frac{(1-q)(2-p)}{(1-p)(2-q)}.\]

Similarly by approximating the process $(Y_t)_{t \ge 0}$ by the process $Y^{(n)}$ and applying
Proposition~\ref{p:ct_rw_improved}, we obtain 
\[
\PP\{ Y_0 \in A, Y_t \in B \} \ge \exp \left(-\frac{(2-e^{-t})(a^2 + b^2)+2e^{-t/2}ab}{4(1- e^{-t})} \right).
\]
Note again that this result holds in an example where the mixing time and $2$-logSob constant are infinite (see \cite{BobLedoux98} where it is shown that the $1$-logSob is finite).


\subsection{Glauber dynamics on Ising model on finite boxes of $\mathbb Z^d$}\label{subset:ising} The Ising model on a finite graph $(V, E)$ has the state space $\Omega = \{ -1, +1\}^V$. The probability of a spin configuration $\sigma \in \Omega$ is given by the Gibbs distribution,
\[ \mu( \sigma)  = \frac{1}{Z(\beta, h)} \exp \left( -\beta \sum_{uv \in E }  \sigma(u) \sigma (v) - h \sum_{ u \in V} \sigma(u) \right),\]
where $Z(\beta, h)$ is the normalizing constant.  The parameters $\beta \ge 0$ and $h$ are called the inverse temperature and the external field respectively.
These definitions extend to infinite locally finite graphs like $\mathbb Z^d$.

The Glauber dynamics for the Ising model is a family of continuous time Markov chains on the state space $\Omega$, reversible with respect to Gibbs distribution, given by the generator
\[ (L f) (\sigma)  = \sum_{ u \in V} c(u, \sigma) ( f(\sigma^u) - f(\sigma)), \]
 where $\sigma^u$ is the configuration $\sigma$ with the spin at $u$ flipped. We consider the two examples of transition rates $c(u, \sigma)$:
 \begin{enumerate}
 \item Metropolis: $c(u, \sigma) = \exp \big(  2 h \sigma(u) + 2 \beta \sigma(u) \sum_{ uv \in E} \sigma(u) \big ) \wedge 1$.
 \item Heat-bath:  $c(u, \sigma) = \left [ 1+ \exp \big(  -2 h \sigma(u) - 2 \beta \sigma(u) \sum_{ uv \in E} \sigma(u) \big ) \right] ^{-1}$.
 \end{enumerate}

 Let $\Lambda := [-n, n]^d \subseteq \mathbb Z^d$ be a finite box in the d-dimensional lattice. Let $\partial_+ \Lambda \subseteq \Lambda^c $ be the vertex boundary of $\Lambda$ in $\mathbb Z^d$. Let $\mu$ be the Gibbs distribution on $\mathbb Z^d$. Given a boundary condition $\tau \in \{ -1, +1\}^{\partial_+ \Lambda}$, we define a Gibbs distribution on $\Lambda$ as a conditional measure:
 \[ \mu_{\Lambda}^\tau = \mu( \cdot|  \sigma_{\partial_+ \Lambda} = \tau ).\]

 Suppose that the inverse-temperature $\beta$ and external field $h$ are such that the Ising model on $\mathbb Z^d$ has strong spatial mixing. Then there exists a constant $K$, independent of $n$,  such that given any boundary condition, the $2$-logSob constant for the Glauber dynamics for the Ising model on the finite box $\Lambda$  is bounded above by $K$, independent of $n$ (see \cite{Martinelli94a, Martinelli94b} which succeed \cite{Stroock92a, Stroock92b, Zegar92} where uniform  bound for $2$-logSob constant was established under stronger Dobrushin-Shlosman mixing conditions). It is also known that in the regime of strong spatial mixing, the mixing time of the Glauber dynamics is $\tm = \Theta(\log n)$.

 The example above can be easily extended to other spin systems and other graphs as long as
 $1$-logSob inequality is established.

\subsection{Random transposition walk on symmetric group}\label{subsec:transpose_walk} The random transposition walk on the group $S_n$ of permutations  of $n$ elements is  the walk generated by the set of all transpositions $\mathcal C_n = \{ (i, j) : 1 \le i < j \le n\}$. The Markov transition from any $\sigma \in S_n$ is described by picking a transposition $\tau$ uniformly at random from $\mathcal C_n$ and compose it with $\sigma$ to get a new permutation $\tau \circ \sigma \in S_n$. It was shown in \cite{Quastel03, BT03, Goel04}  that the 1-logSob constant $C$ of this chain is of order $n$. More precisely,
\[ \frac{n-1}{2} \le C \le 2(n-1). \]
On the other hand, it is well known (see \cite{Diaconis81}) that the mixing time $\tm = \Theta(n \log n)$.  It's worth mentioning that the $2$-logSob constant of the random transposition walk was determined in \cite{Lee98} to satisfy $C' = \Theta(n \log n)$.

\subsection{Top-to-random transposition walk on symmetric group}\label{subsec:top_to_random} This is a random walk on $S_n$ generated by the set of transpositions  $\mathcal D_n = \{ (1, j) : 2 \le j \le n\}$. Again the 1-logSob constant $C$ of this chain satisfies \cite{Goel04}
\[ \frac{n-1}{2} \le C \le 2(n-1), \]
whereas the mixing time is $\tm = \Theta(n \log n)$ (see \cite{Diaconis92}).

\subsection{Random walk on spanning trees}\label{subsec:spanning} This is a natural random walk on the space of all spanning trees of  a graph $G = (V, E)$. Suppose $T$ be our current spanning tree. We choose an edge $e \in E$ and another edge $f \in T$ uniformly at random. If $T' = T \cup \{ e\} \setminus \{f\}$ is a spanning tree of $G$, we update $T$ to $T'$, otherwise we remain at $T$. It was shown in \cite{Jerrum02} that the $2$-logSob constant of this walk satisfies
\[ C \le |V| |E|, \]
and consequently, $\tm = O(|V| |E| \log |V|)$.  In general, the upper bound for the mixing time is tight. For example, consider a line of length $n$ and replace each edge by a double edge. Thus the new graph has $|v| = n+1$ and $|E| = 2n$. The mixing time for the random walk on the spanning trees of this graph is same as the coupon collector problem with a delay of $\Theta(n)$ between successive moves.

\subsection{Bernoulli-Laplace model}\label{subsec:BL} This is natural random walk on the subsets of size $r$ of the ground set $\{1, 2, \ldots, n\}$, $1 \le r < n$. So, the state space has size ${ n \choose r}$. If the current state of Markov chain is an $r$-set $A$, we pick an element $i$ uniformly at random from $A$ and pick an element $j$ uniformly at random from $\{1,2,\ldots, n\} \setminus A$ and switch the elements to obtain a new $r$-set $A' = A \cup \{ j\} \setminus \{i\}$. This is also known as simple exclusion process on the complete graph on $n$ vertices. The $1$-logSob constant of this chain satisfies \cite{Quastel03, BT03, Goel04}
\[  \frac{r(n-r)}{2n} \le C \le \frac{2r(n-r)}{n}.\]
The mixing time for Bernoulli-Laplace model is $\tm = O( \frac{r (n-r)}{n} \log \log { n \choose r}  )$.

\section{A quantitative Arrow theorem for general ranking distributions}
Our goal in this section is to prove Theorem~\ref{thm:maingeneralintro}. We begin by briefly introducing some additional notation. Let $A = \{a,b,\ldots,\}$ be a set of $k \geq 3$ alternatives. A {\em transitive preference} over $A$ is a ranking of the alternatives from top to bottom where ties are not allowed. Such a ranking naturally corresponds to a {\em permutation} $\sigma$ of the elements $1,\ldots,k$. The group of all rankings will be denoted by $S_k$.
A {\em constitution} is a function $F$ that associates to every $n$-tuple $\sigma = (\sigma(1),\ldots,\sigma(n))$ of transitive preferences, and every pair of alternatives $a,b \in A,$ a (strict) preference between $a$ and $b$. Some key properties of constitutions include {\em Transitivity}, {\em Independence of Irrelevant Alternatives (IIA)}, {\em Unanimity} (all defined at the introduction). 
Recall that the constitution $F$ is a {\em dictator} on voter $j$, if $F(\sigma) = \sigma(j)$, for all $\sigma$, or 
$F(\sigma) = \sigma(j)^{-1}$, for all $\sigma$, where $\sigma(j)^{-1}$ is the inverse of the permutation $\sigma(j)$. 

We will assume each voter chooses one ranking from $S_k$ according to some fixed distribution $\varrho$, independently of others. We will write $\PP$ for the product measure $\varrho^{\otimes n}$ on $S_k^n$
and $\E$ for the corresponding expected value.
We now quickly sketch how one can prove Theorem~\ref{thm:maingeneralintro} with the explicit bound 
\begin{equation}\label{delta_dep_eps}
 \delta = \exp \left ( - \frac{C \alpha^{-7}2^{\alpha^{-2}}  ( \log(1/\eps))^2}{\eps^{2+\frac{1}{2\alpha^2}}}\right).
 \end{equation}


We begin with some notation and definitions from ~\cite{Mossel11}.

Given $\sigma = (\sigma(1),\ldots,\sigma(n)) \in S_k^n$ and for each pair of alternatives $a, b \in A$,  we define binary vectors $x^{a>b} = x^{a>b}(\sigma)$ in the following manner:
\[
 x^{a>b}(j) = 1, \quad \mbox{if voter } j \mbox{ ranks } a \mbox{ above } b; \]
 and
 \[ x^{a>b}(j) = -1, \quad \mbox{if voter } j \mbox{ ranks } b \mbox{ above } a.
\]
 Thus, if $F$ satisfies the IIA property then there exist functions $f^{a>b}$ for every pair of candidates $a$ and $b$ such that
\[
F(\sigma) = ((f^{a>b}(x^{a>b}) : \{a,b\}, a \ne b \in A )
\]
where $f^{a>b} : \{-1, 1\}^n \to \{-1, 1\}$ is such that $f^{a>b}=+1$ if $F$ ranks $a$ over $b$ and $f^{a>b}=-1$ otherwise and where we have $f^{a>b}(x) = -f^{b>a}(x)$ for all $a,b$ and all $x$. 

We {define} $\Px(f_1,f_2,f_3)$ ($\Px$ stands for paradox) for three function $f_1,f_2,f_3 : \{-1,1\}^n \to [-1,1]$ by letting
\begin{align*}
\Px(f_1,f_2,f_3)=
\frac{1}{4}\Big( 1 + &\E[f_1(x^{a>b})f_2(x^{b>c})] + \E[f_2(x^{b>c})f_3(x^{c>a})] \\
+ &\E[f_3(x^{c>a})f_1(x^{a>b})] \Big) .
\end{align*}
Note that for $k=3$, the probability of non-transitive outcome is given by
\begin{align*}
P(F) &:= \PP \big \{ (f^{a>b},f^{b>c},f^{c>a}) \in \{ (1,1,1), (-1,-1, -1) \} \big\}  \\
&= \Px(f^{a>b},f^{b>c},f^{c>a}).
\end{align*}

In the rest of the subsection, we denote by $\alpha$, the probability mass of smallest atom of the distributions of the random vectors  $(x^{a>b}(1), x^{b >c}(1), x^{c>a}(1))$  on $\{-1, 1\}^3$ for triplets of distinct alternatives $a, b, c\in A$.


We now quickly discuss the notions of influences that are needed in the proof.
For a function $f: \{-1, 1\}^n \to \mathbb R$ where $\{-1, 1\}^n$ equipped with $n$-fold product  of  some biased measure $\mu_p = (1-p) \delta_{-1} + p \delta_1$, define the {\em influence}
of variable $i$ on $f$ by
\[\I_i(f) :=  \mu_p^{\otimes n} \big\{ f(x_1, \ldots, \underbrace{-1}_{i}, \ldots, x_n)  \ne f(x_1, \ldots, \underbrace{+1}_{i}, \ldots, x_n) \big\}.\]

Let $\{ \psi_0 \equiv 1, \psi_1 \}$ form a basis of $L^2( \{-1,1\}, \mu_p )$. Then we can express $f$ in its Fourier basis as follows:
\[ f(x) = \sum_{ S \subseteq [n]} \hat f(S) \prod_{ i \in S} \psi_1(x_i).\] We define variance-influence of variable $i$ on $f$ as
\[ I_i(f)  :=  \sum_{S :  i \in S}  \hat f(S)^2. \]
When $f$ is $\pm 1$-valued, it can be easily checked that the above two notions of influences are equivalent up to a multiplicative factor (independent of $n$) as follows: \[ I_i(f) \le  \I_i(f) \le  \frac{1}{4 p(1-p)} I_i(f).\]
We also need the notion of low-degree variance-influences. For $d >0$,  this is defined as follows:
\[ I^{ \le d}_i(f) :=  \sum_{S :  i \in S, |S| \le d}  \hat f(S)^2.\]

 Under our assumption on the minimum atom of $\varrho$, it's not difficult to show that  for any three distinct alternatives $a, b, c \in A$ and any voter $i$, we have
  \[ | \mathrm{Corr}(x^{a> b}(i), x^{b>c}(i)) | \le 1 - 4\alpha.\]

The following lemma   is a consequence of the reverse hypercontractivity in the biased space. It is the key ingredient  needed to extend the argument of \cite{Mossel11} to the nonuniform case.  
\begin{lemma} \label{lem:two_inf}
Consider a social choice function on $3$ candidates $a,b$ and $c$ and $n$ voters denoted $1,\ldots,n$. Assume that the social choice function satisfies that IIA condition and that voters vote independently according to $\varrho$ whose atoms are  bounded below by a constant $\alpha>0$. Assume further that
$\I_1(f^{a>b}) > \eps$ and $\I_2(f^{b>c}) > \eps$.
Let
\[
A = \{\sigma : 1 \mbox{ is pivotal for } f^{a>b} \}, \quad B = \{\sigma : 2 \mbox{ is pivotal for } f^{b>c} \}.
\]
Then
\[
\PP\{ A \cap B\} \geq  \eps^{\frac{2- \sqrt{1-\alpha}}{1 - \sqrt{1-\alpha}}}.
\]
\end{lemma}
Here voter $j$ is called `pivotal' for $f^{a>b}$ at  $\sigma$ if  $f^{a>b}$ is a non-constant function of the $j^{th}$ variable when we freeze the other $(n-1)$ variables at $x^{a< b}(\sigma)$.

\begin{proof}
 Clearly, $(x^{a>b}(i), x^{b>c}(i))_{ 1 \le i \le n} $ are i.i.d.\ with a joint distribution on $\{-1, 1\}^2$ determined by $\varrho$. Let $\mu$ and $\nu$ be the marginal distributions of $x^{a>b}(i)$ and  $x^{b>c}(i)$ respectively. Note that the event $A$ is determined by $x^{a>b}$ and the event $B$ is determined by $x^{b>c}$ and the their intersection probability is determined
by the joint probability distribution of the random vectors $x^{a<b}$ and $x^{b<c}$.  Let $A_0$ and $B_0$ be the subsets of $\{-1, 1\}^n$ defined by:
\[ A_0 = \{ x^{a<b} \in \{-1, 1\}^n: f^{a<b}(x^{a<b}) \ne  f^{a<b}(x^{a<b} e_1) \}, \]
and \[  B_0 = \{ x^{b<c} \in \{-1, 1\}^n: f^{b<c}(x^{b<c}) \ne  f^{b<c}(x^{b<c}  e_2) \},\]
where $e_1 = (-1,1,\ldots,1)$ and $e_2 = (1,-1,1,\ldots,1)$ so that 
$(x_1,\ldots,x_n) e_1 = (-x_1,x_2,\ldots,x_n)$ and $(x_1,\ldots,x_n) e_2 = (x_1,-x_2,x_3\ldots,x_n)$.

Observe that  $\mu^{\otimes n}\{A_0\} = \I_1(f^{a>b}) > \eps$ and $\nu^{\otimes n}\{B_0\}  = \I_2(f^{b>c}) > \eps$, and our goal is to obtain a bound on $\PP\{A_0 \cap B_0\}$. But now we are exactly in the set-up of Lemma~\ref{l:corr_sets_noise_cube}.
If $K$ denotes the conditional distribution of $x^{b<c}(i)$ given $x^{a<b}(i)$, then we have the following lower bound on
\[ \min_{u, v \in \{ -1, 1\} } \frac{K(u, v)}{\nu\{v\}} = \min_{u, v \in \{ -1, 1\} } \frac{\PP \{ x^{a<b}(i) = u, x^{b<c}(i) = v \}}{\mu\{u\}\nu\{v\}}  \ge \alpha.  \]
The proof now follows from Lemma~\ref{l:corr_sets_noise_cube}.
\end{proof}

The reminder of the proof is a straightforward (though somewhat tedious) generalization of the proof given in~\cite{Mossel11} that does not use reverse hypercontractivity.  A sketch of the modifications needed is given in Appendix~\ref{a:arrow}.

\section{Non-interactive correlation  distillation for dice}

The proof of Theorem~\ref{thm:nicd} is a generalization of the proof given in~\cite{Mossel06}.
The proof of the upper bound uses reverse hypercontractivity while the lower bound is based on the analysis of a simple protocol that is based on the plurality function and the analysis relies the on normal approximation. Here we give the proof of the upper bound. The proof of the lower bound is an easy (if tedious) adaptation of \cite{Mossel06} and is given in Appendix~\ref{a:nicd}.

\begin{proof}[Proof of upper bound of Theorem~\ref{thm:nicd}]  Note that the probability of all players output $j \in \Omega$ is
\begin{equation}
\label{eq:dice_upp_bnd}
\E \left[ \prod_{i=1}^k \PP \{F_i(y) = j| x\} \right ],
\end{equation}
where $y$ is a $\rho$-correlated copy of $x$. Let $f_{i, j}(x) := \1_{ \{F_i(x) = j \}}$. Thus  if
$t=\log(1/\rho)$ and $T_t$ is the simple semigroup on $\Omega$ with the uniform measure,  then we have
\[  \text{\eqref{eq:dice_upp_bnd}} =  \E \left[ \prod_{i=1}^k \E [f_{i, j}(y) | x] \right ]  =   \E \left[ \prod_{i=1}^k T_t f_{i, j}(x)  \right ]  \le   \prod_{i=1}^k  \|  T_t f_{i, j}\|_k, \]
where the last step follows from the H\"{o}lder's inequality. Since $\E f_{i, j} = m^{-1}$ for all $i, j$,  we conclude, by Lemma~\ref{lem:power} below, that the probability of total agreement among $k$ players is bounded above  by $\sum_{j \in \Omega} \prod_{i=1}^k  \|  T_t f_{i, j}\|_k \le Cm k^{-\gamma_1}$ for some $\gamma_1>0$ depending on $\rho$.
\end{proof}

\begin{lemma} \label{lem:power}   Let $\Omega  = \{1, 2, \ldots, m\}$ and $\mu$ be the uniform measure on $\Omega$. Fix any $\rho \in (0,1]$.  Then there exist constants $C = C(\rho)>0, \beta = \beta(\rho) >0$  such that for any $f : \Omega^n \to [0,1]$ and for any $k \ge 1$ such that $\E f \le 1/2$,
\[ \| T_t f \|_k^k  \le Ck^{-\beta}.  \]
\end{lemma}

\begin{proof}[Proof of Lemma~\ref{lem:power}]
Suppose $  \| T_t f \|_k^k \ge 2\delta$. Define $S = \{ x \in \Omega^n :  [T_t f(x)]^k \ge \delta \}$. Since $T_t f$ is bounded between $0$ and $1$, it follows that $\E [1_S]  \ge \delta$. If we write $\cf$  for $1- f$, then the set $S$ has the following equivalent description
\[ S = \{  x \in \bn :  T_t \cf(x) \le 1 - \delta^{1/k} \}.\]
Thus clearly we have
\begin{equation}  \label{eq:power_ub}
\E [1_S T_t \cf] \le (1 - \delta^{1/k}) \PP\{ S\}.
\end{equation}
On the other hand,  Corollary~\ref{cor:two_function} gives us that
\begin{equation*}
\E [ 1_S T_t \cf] \ge  \| 1_S \|_p \| \cf \|_q \quad \text{ for any } 0 < p, q<1 \text{ satisfying } \rho \le \frac{4(1-p)(1-q)}{(2-p)(2-q)}.
\end{equation*}
If we take $ p = q = \frac{2(1 - \sqrt \rho)}{2 - \sqrt \rho}$ in the above inequality, we have
\begin{equation}  \label{eq:power_lb}
\E [1_S T_t \cf] \ge  \PP \{ S\}^{ \frac{ 2 - \sqrt \rho} {2(1- \sqrt \rho)}} (\E \cf)^{\frac{ 2 - \sqrt \rho} {2(1- \sqrt \rho)}}, \end{equation}
where we have used the fact that $ \E[ \cf^q ]\ge \E \cf$ for any $q \in (0, 1)$. Now, comparing  \eqref{eq:power_ub} and  \eqref{eq:power_lb}, we have
\[ \PP \{ S\}^{ \frac{  \sqrt \rho} {2(1- \sqrt \rho)}}   (\E \cf)^{\frac{ 2 - \sqrt \rho} {2(1- \sqrt \rho)} } \le 1 - \delta^{1/k}.\]
Since, $\PP \{S\} \ge \delta$ and $ \E \cf \ge 1/2$, we have
\[ \delta^{ \frac{  \sqrt \rho} {2(1- \sqrt \rho)}}  2^{- \frac{ 2 - \sqrt \rho} {2(1- \sqrt \rho)} } \le 1 - \delta^{1/k},\]
which implies that $\delta  \le  k^{-\beta}$ for any  $0< \beta <  \frac{2(1- \sqrt \rho)} {  \sqrt \rho} $ and $k$ sufficiently large.
\end{proof}
\begin{remark}
 It is an interesting problem to find the exact exponent $\gamma$ in Theorem~\ref{thm:nicd}  for which   $ \lim_{ n \to \infty} \mathcal M_\rho(k, n) = k^{-\gamma+o(1)}$ as $k \to \infty$. A priori such an exponent might depend on $m$. 
\end{remark}

\section{Observations and open problems}

Our main result on the monotonicity of $r$-logSob inequalities implies that the Poincar\'e ($0$-logSob) inequality  is the weakest among them.

However several open problems regarding monotonicity:
\begin{enumerate}
\item[(I)]
Are there intervals $I$ such that $r$-logSob inequalities are equivalent for all reversible Markov semigroups and all $r \in I$. In other words, for which intervals
$I$, there exist constants $c(I)$ such that for all $r,s \in I$, 
$r$-logSob with constant $C$ implies $s$-logSob with constant $c(I) C$?  
Note that Proposition~\ref{reversinglogSob} implies a positive answer to this question with the interval $[1+\epsilon,2]$ for any $\eps>0$. Note that this interval can not be extended to $[1,2]$. This follows for example from the fact  that for the random transposition card shuffling on the symmetric group $S_n$,   $2$-logSob constant is $\Theta(n \log n)$ \cite{Lee98} whereas $1$-logSob constant is known to be $\Theta(n)$ \cite{Quastel03, Goel04}. 

\item[(II)]
Can one establish similar monotonicity property for hypercontractive inequalities?
\end{enumerate}

\subsection{Reverse hypercontractivity implies spectral gap}

Here we show that the Poincar\'e inequality may be deduced from reverse hypercontractivity for {\em fixed} $q<p<1$. This provides a partial answer to question (II) above.

\begin{lemma} \label{revhyp->poin}
Let $q<p<1$ and $t>0$. Assume that a symmetric Markov semigroup satisfies the reverse hypercontractivity
estimate $\| T_{t}f\|_{q} \geq \| f\|_{p}$ for every  $f \in \H_{(0,\infty)}$. Then it also satisfies the Poincar\'e inequality
\[
\Var(g) \leq \frac{2t}{\log(1-q)-\log(1-p)} \cdot \e(g,g)
\]
for every $g \in \H$.
\end{lemma}

\begin{proof}
Let $\lambda=\inf \sigma(L|1_{\perp})$, so that $\e(g, g) \ge \lambda \Var(g)$ for all $g$ belonging to
$1^{\perp}$,  where $1^{\perp}$ denotes the $L$-invariant subspace of $\H$ consisting of all functions
orthogonal (in the standard $L^{2}(\Omega, \mu)$ setting) to the constant function $1$, i.e.,  zero-mean functions.
For a zero-mean $g \in \H$ choose $\varepsilon>0$ small enough to make $f=1+\varepsilon g>0$. The inequality
\[
(\| T_{t}f\|_{q}-1)/\varepsilon^{2} \geq
(\| f\|_{p}-1)/\varepsilon^{2}
\]
upon passing to the limit $\varepsilon \to 0^{+}$ yields
$\E [ge^{-2tL}g]=\E[ (T_{t}g)^{2}] \leq \frac{1-p}{1-q} \E[g^{2}]$. Since this bound holds for all $g \in 1^{\perp}$ we infer that
$e^{-2t\lambda} \leq (1-p)/(1-q)$ which ends the proof.
\end{proof}

\subsection{Spectral gap does not imply $1$-logSob}
Here we show that the $0$-logSob inequality does not imply the $1$-logSob inequality. In particular it gives a partial answer to question (I) above
by showing that the $r$-logSob inequalities in the interval $[0,1]$ are not all equivalent. 
Recall that a family graphs $\mathcal{G} = \{ G_1, G_2, \ldots \}$ is called a $d$-regular (spectral) expander if
\begin{enumerate}
\item  For each $n$, $G_n =(V_n, E_n)$ is a $d$-regular graph on $n$ vertices.
\item  The random walk on $G_n$ satisfies a Poincar\'{e} inequality with constant $C_0$ that does not depend on $n$.
\end{enumerate}
Assume, by way of contradication that there exists a constant $C_1$ such that for each $n$, $1$-logSob constants for the random walks on $G_n$ are bounded above by $C_1$.  Let $(X^n_t)_{ t \ge 0}$  be the continuous-time random walk on $G_n$. Since the underlying graph is $d$-regular, the stationary distribution $\pi$ is the uniform measure on $G_n$. So, if we take $A=\{u\}$ and $B = \{ v\}$ for $u, v \in V_n$ in \eqref{ieq:twosetAB}, then  we have
\[  \PP \{ X^n_0 =u, X^n_1 = v\} \ge n^{ - \alpha} \quad \forall u, v \in V_n, \]
where $\alpha >0$ is a constant that depends on $C_1$. This implies that
\begin{equation}\label{eq:uvprob}
 \PP^u\{ X^n_1 = v\} \ge n^{ - \alpha +1} \quad \forall u, v \in V_n.
 \end{equation}
 Clearly the diameter of the graph $G_n$ has to be at least  $c \log n / \log d$ for some constant $c>0$. We choose $u, v \in V_n$ so that their graph distance is at least $c \log n /\log d$. So, starting from $u$, a discrete time random walk on $G_n$ needs at least $c \log n / \log d$ many jumps before it can reach the vertex $v$. Note that the number of jumps made by the continuous-time walk $X^n_t$ during the time interval $[0,1]$ is distributed according to $\mathrm{Poisson}$ random variable with mean $1$. Hence,  from the tail bound for the Poisson distribution,
\[  \PP \{ X_0 =u, X_1 = v\} \le \PP\{ \mathrm{Poisson}(1) \ge c \frac{\log n}{\log d} \} \le 
n^{-\frac{c'}{\log d}(\log \log n - \log \log d)}
  \]
for some constant $c'>0$. Since the right hand side of the above inequality decays faster than any polynomial, it contradicts
\eqref{eq:uvprob}. This proves that $1$-logSob constant for the random walk on $G_n$ tends to infinity as $n \to \infty$.

\begin{remark}
Explicit lower bounds on  $1$-logSob constants for connected $d$-regular graphs on $n$ vertices   can  be found in \cite{Goel04, BobkovTetali06}. But our proof is different in the sense that it relies on the new mixing bounds implied by reverse hypercontractivity. 

\end{remark}
\subsection{Generalizations to infinite spaces}
It is straightforward to generalize most of the result of Sections 1-9 of the paper to infinite probability spaces.
The only point which requires some care is to work with the appropriate classes of functions.
Since  the applications in the current paper deal mainly with finite spaces we omit this straightforward extension.

\bibliographystyle{amsalpha}
\bibliography{revhyp}

\appendix


%


\section{Proof of Theorem~\ref{thm:maingeneralintro}} \label{a:arrow}
We continue in the proof of the  general quantitative Arrow theorem following~\cite{Mossel11}.

The next step is to replace Theorem~7.1 and Theorem~11.11 in \cite{Mossel11}
by the following two lemmas respectively.

\begin{lemma} \label{thm:arrow_one_inf}
For every $\eps > 0$ there exist $\delta(\eps) > 0$ and  $\tau(\delta) > 0$ such that the following hold.
Let $f_1,f_2,f_3 : \{-1,1\}^n \to \{-1,1\}$ and let $F$ be the social choice function defined by 
$f^{a>b} = f_1, f^{b>c} = f_2$ and $f^{c>a} = f_3$.
Assume that for all $1 \leq i \leq 3$ and $j > 1$ it holds that
\begin{equation} \label{eq:inf5}
I^{\le\log(1/\tau)+1}_j(f_i) < \alpha \tau/2.
\end{equation}
Then either
\begin{equation} \label{eq:paradox5}
\Px(f_1,f_2,f_3) \geq \alpha \delta,
\end{equation}
or there exists a social choice function $G$ which is either a dictator or always ranks one candidate at top/bottom such that $D(F,G) \leq 9 \eps$.
Moreover, one can take
\[
\delta = \frac{1}{4}(\eps/2)^{2+1/(2 \alpha^2)}, \quad \tau = \delta^{C \frac{\log (2/\alpha)}{\alpha} \frac{\log(1/\delta)}{  \delta }}.
\]
\end{lemma}

\begin{lemma} \label{thm:arrow_low_cross}
For every $\eps > 0$, there exist  $\delta(\epsilon) > 0$ and  $\tau(\delta) > 0$
such that the following hold.
Let $f_1,f_2,f_3 : \{-1,1\}^n \to [-1,1]$.
Assume that for all $1 \leq i \leq 3$ and all $u \in \{-1,1\}$ it holds that

\begin{equation} \label{eq:non_top_bot_cross}
\min(u \E[f_i], -u \E[f_{i+1}]) \leq 1-3\eps \quad (\text{with convention } f_4= f_1)
\end{equation}
and for all $1 \le j \le n$ it holds that
\begin{equation} \label{eq:cross_inf}
|\{ 1 \leq i \leq 3 : I_j^{\le (\log(1/\tau))^2}(f_i) > \tau \}| \leq 1.
\end{equation}
Then we have
\[
\Px(f_1,f_2,f_3) \geq \delta.
\]
Moreover,  one can take:
\[
\delta = \frac{1}{8}(\eps/2)^{2+1/(2 \alpha^2)}, \quad \tau = \delta^{C \frac{\log (2/\alpha)}{\alpha} \frac{\log(1/\delta)}{  \delta }}.
\]
\end{lemma}

The proofs of the above two lemmas are almost identical to those given in \cite{Mossel11}.  The only difference is that instead of  Theorem~11.10 of \cite{Mossel11} we now use its modified version as follows.
\begin{lemma} \label{thm:arrow_low_inf}
For every $\eps > 0$, there exist $\delta(\eps) > 0$ and  $\tau(\delta) > 0$ such that the following hold.
Let $f_1,f_2,f_3 : \{-1,1\}^n \to [-1,1]$.
Assume that for all $1 \leq i \leq 3$ and all $u \in \{-1,1\}$ it holds that
\begin{equation} \label{eq:non_dict_low_inf}
\min(u \E[f_i], -u \E[f_{i+1}]) \leq 1-3\eps \quad (\text{with convention } f_4= f_1)
\end{equation}
and for all $1 \leq i \leq 3$ and $1 \leq j \leq n$ it holds that
\[
I_j^{\log(1/\tau)}(f_i) < \tau,
\]
Then we have
\[
\Px(f_1,f_2,f_3) > \delta.
\]
Moreover, one can take:
\[
\delta = \frac{1}{4}(\eps/2)^{2+1/(2 \alpha^2)}, \quad \tau = \delta^{C \frac{\log (2/\alpha)}{\alpha} \frac{\log(1/\delta)}{  \delta }}.
\]
\end{lemma}

The proof of  Lemma~\ref{thm:arrow_low_inf} depends on  Gaussian Arrow's theorem (see Theorem~11.7 of \cite{Mossel11})  and the following generalization of some Gaussian invariance result proved in \cite{Mossel11} (see Theorem 11.9). The latter may be of independent interest.


\begin{lemma}[Invariance principle] \label{thm:invariance}
Let $\eps >0, -1 < \rho < 1$.
Then for every measurable function $f : \{-1,1\}^n \to [-1,1]$ there exists a measurable function $\widetilde{f} : \R^n \to [-1,1]$ such that the following holds for any $n \ge 1$.
Let  $(X,Y)$ be distributed on $\{-1,1\}^n \times \{-1,1\}^n$ where $(X_i,Y_i)_{ 1 \le i \le n}$ are i.i.d.\  with
$\mathrm{Corr}(X_i ,Y_i) = \rho.$
Let $\gamma>0$ be a lower bound for the smallest atoms of the random variables $X_i$ and $Y_i$ on $\{-1, 1\}$.
Consider $(N,M)$ jointly Gaussian and distributed in $\mathbb R^n \times \R^n$ with $(N_i,M_i)_{ 1 \le i \le n}$ are i.i.d.\  with
\[
\E[N_i] = \E[M_i] = 0, \quad \E[N_i^2] = \E[M_i^2] = 1, \quad \E[N_i M_i] = \rho.
\]
Then
\begin{itemize}
\item
For the constant functions $1$ and $-1$ it holds that $\widetilde{1} = 1$ and $\widetilde{ -1} = -1$.
\item
If $f$ and $g$ are two functions such that for all $1 \le i \le n$, it holds that \[ \max(I_i^{\le \log(1/\tau)}(f),I_i^{\le \log(1/\tau)}(g)) < \tau, \] then
\begin{equation} \label{eq:inv}
\big|\E[f(X) g(Y)] - \E[\widetilde{f}(N) \widetilde{g}(M)] \big| \leq \eps,
\end{equation}
whenever
\begin{equation} \label{eq:tau_bound}
\tau \leq \eps^{C \frac{ \log(2/\gamma)}{(1-|\rho|)} \cdot \frac{\log(1/\eps)}{\eps}},
\end{equation}
for some absolute constant $C>0$.
\end{itemize}
\end{lemma}
\begin{proof}[Proof of Lemma~\ref{thm:invariance}]
The proof is same as Theorem~11.9 of \cite{Mossel11}. The only difference is that we now need to apply the version of Theorem~3.20 in~\cite{MoOdOl:10} under hypothesis H3 instead of hypothesis H4.
\end{proof}

\begin{proof}[Proof of Theorem~\ref{thm:maingeneral}]
 We will only give a brief sketch the proof of theorem for $k=3$. The proof for $k > 3$  follows from a general argument given in \cite{Mossel11}.

 Take $\tau = \tau(\eps) = \delta_0^{C \log (2/\alpha) \frac{\log(1/\delta_0)}{ \alpha \delta_0 }}, \delta_0= \frac{1}{8}(\eps/2)^{2+1/(2 \alpha^2)}$ as in Lemma~\ref{thm:arrow_low_cross}  and $\eta = \alpha \tau/2$.

Let $f^{a>b},f^{b>c},f^{c>a} : \{-1,1\}^n \to \{-1,1\}$ be the three pairwise preference functions.
Let $\eta = \delta$ (where the values of $C$ will be determined later).
We will consider three cases:
\begin{itemize}
\item
There exist two voters $i \neq j \in [n]$ and two functions $f \neq g \in \{f^{a>b},f^{b>c},f^{c>a}\}$ such that
\begin{equation} \label{eq:case1}
I^ { \le (\log(1/\tau))^2}_i(f) > \eta, \quad I^ { \le (\log(1/\tau))^2}_j(g) > \eta.
\end{equation}
\item
For every two functions $f \neq g \in \{f^{a>b},f^{b>c},f^{c>a}\}$ and every $i \in [n]$, it holds that
\begin{equation} \label{eq:case2}
\min(I^ { \le (\log(1/\tau))^2}_i(f),I^ { \le (\log(1/\tau))^2}_i(g)) < \eta.
\end{equation}
\item
There exists a voter $j'$ such that for all $j \neq j'$
\begin{equation} \label{eq:case3}
\max(I^ { \le (\log(1/\tau))^2}_j(f^{a>b}),I^ { \le (\log(1/\tau))^2}_j(f^{b>c}),I^ { \le (\log(1/\tau))^2}_j(f^{c>a})) < \eta.
\end{equation}
\end{itemize}

First note that each $F$ satisfies at least one of the three conditions (\ref{eq:case1}), (\ref{eq:case2})
or (\ref{eq:case3}). Thus it suffices to prove the theorem for each of the three cases.

In~(\ref{eq:case1}), we have $\I_i(f) > \eta $ and $ \I_j(g) > \eta$. By Lemma~\ref{lem:two_inf} combined with Barbera's Lemma \cite{barbera80} (see Proposition 3.1 of \cite{Mossel11}), we obtian
\[
P(F) > \alpha^2 \eta^{\frac{2- \sqrt{1-\alpha}}{1 - \sqrt{1-\alpha}}} \ge  \alpha^2 \eta^{\frac{4}{\alpha}} .
\]
We thus obtain that $P(F) > \delta$ where $\delta$ is given in~(\ref{delta_dep_eps})
by taking large $C$.

In case~(\ref{eq:case2}), by Lemma~\ref{thm:arrow_low_cross}, it follows that Either (if~(\ref{eq:non_top_bot_cross}) does not hold) there exists a function $G$ which
always puts a candidate at top/bottom and $D(F,G) < 3\eps$, Or,
$P(F) >  \frac{1}{8}(\eps/2)^{2+1/(2 \alpha^2)} \gg \delta$.

Similarly in the remaining case~(\ref{eq:case3}), we have by Lemma~\ref{thm:arrow_one_inf} that Either $D(F,G) < 9\eps$ Or
$P(F) > \frac{1}{4}(\eps/2)^{2+1/(2 \alpha^2)} \gg \delta$. The proof follows.
\end{proof}
\begin{remark}
Keller  ~\cite{Keller10} proved  that one may take $\delta = C \eps^3$
in the special case when $\varrho$ is {\em uniform}. It's an interesting open question to see whether such polynomial dependence of $\delta$ on $\eps$ holds for general distribution $\varrho$.
\end{remark}

\section{A lower bound for the NICD problem using a plurality function} \label{a:nicd}
\begin{proof}[Proof of the lower bound for the NICD problem]
We will analyze the protocol where all players use some balanced plurality function $\PL_n$ that we are going to described below.
Define $n_j =\#\{ i :x_i = j \}$ to be the number of times $j$ is present in the string $x$ and set $R = \{ j \in \Omega : n_j = \max_{l \in \Omega} n_l \}$. Then we define our pluraity function as
\[ \PL_n(x) = x_{i*} \ \ \text{ where } i_* =\min\{i : x_i \in R \}.    \]
Note that if $j$ is the unique value in $\Omega$ which occurs most frequently in string $x$, that is, if $R = \{j\}$,  then $\PL_n(x) = j$. Also, note that if $\sigma$ is any permutation of $\Omega$ then
\[ \PL_n(\sigma(x_1), \sigma(x_2), \cdots, \sigma(x_n)) = \sigma (\PL_n(x)),\]
which implies that $\PL_n$ is balanced.

Define $W_{j} = W_{j}^{(n)} :=  n^{-1/2} \sum_{ i=1}^n (\1_{\{x_i = j \}} - m^{-1})$ and $W'_{j} = {W'_{j}}^{(n)} :=  n^{-1/2} \sum_{ i=1}^n (\1_{\{y_i = j \}} - m^{-1})$ where $y$ is a $\rho$-correlated of $x$.

The probability of total agreement among $k$ players is bounded below by the probability event that they all output $1$ which is at least
\begin{equation}\label{eq:lwbnd_dice}
  \E \left [ \PP\{ W'_1  >  \max_{ j\ne 1} W'_j | W_j, 1\le j \le m \}^k \right].
  \end{equation}

Next we proceed to bound $\liminf_n \PP\{ W'_1  >  \max_{ j\ne 1} W'_j | A \}$ where
$ A = \{ W_1 \ge  \tfrac{2a}{\rho}  \text{ and } W_j \le  0 \text{ for all } 2 \le j \le m \}$ and $a = a(k, m) = \tfrac{\sqrt{2 \log(km)}}{\sqrt{m}}$.
Note that
\begin{align*}
&\PP\{ W'_1  >  \max_{ j\ne 1} W'_j | A\} \ge \PP\{ W'_1  >  a | A\}  - \sum_{j=2}^m \PP\{ W'_j  \ge   a  | A\} \\
 &= \PP\{ W'_1  > a   | W_1 \ge   \tfrac{2a}{\rho}   \}  - \sum_{j=2}^m \PP\{ W'_j  \ge   a  | W_j \le 0\}.
\end{align*}
The last step is justified by the fact that $W_j$ is a sufficient statistics for the conditional distribution of $W_j' $ given $x$.

Note that for all $1 \le j \le m$
\[  \E\1_{\{x_i = j \}}  = \E (\1_{\{y_i = j \}}  =m^{-1}, \ \  \mathrm{Var}(\1_{\{x_i = j \}})  =   \mathrm{Var}(\1_{\{y_i = j \}}) = m^{-1}(1- m^{-1})\]
and for all $1 \le j \ne j' \le m$
\[ \mathrm{Cov} (\1_{\{x_i = j \}} , \1_{\{y_i = j \}}) =  \rho m^{-1}(1 -m^{-1}),   \mathrm{Cov} (\1_{\{x_i = j \}} , \1_{\{x_i = j' \}}) =  - m^{-2}. \]
 It now follows from multidimensional Central Limit Theorem that
\[  (W_j, W'_j) \stackrel{d}{\to} N_2(0, \Sigma)\]
and
\[ (W_j, 1\le j \le m) \stackrel{d}{\to} N_m(0, \Gamma) \]
as $n \to \infty$, where $N_2(0, \Sigma)$ (resp.\ $N_m(0, \Gamma)$)   is the two-dimensional  (resp.\ $m$-dimensional)  normal distribution with mean zero and covariance matrix $\Sigma$ (resp.\ $\Gamma$) given by
\[ \Sigma = m^{-1}(1  - m^{-1})  \begin{pmatrix}
 1 & \rho \\
\rho &1 \\
\end{pmatrix}
\text{ and }  \Gamma = m^{-1} I_m - m^{-2} \1\1'
.\]
Moreover, for any any convex regions $R_1 \subseteq \mathbb R^2$ and $R_2 \subseteq \mathbb R^m$, we have the Berry-Ess\'een-type error bound \cite{Sazonov81}  as the following:
\begin{equation} \label{conv:berryesseen2}
 \left| \PP\{ (W_j , W_j') \in R_1\} - \PP\{ (Z_1, Z_2) \in R_1\} \right | = O(n^{-1/2}),
 \end{equation}
and
\begin{equation} \label{conv:berryesseenm}
\left| \PP\{ (W_j , 1 \le j \le m) \in R_2\} - \PP\{ (X_j, 1 \le j \le m) \in R_2\} \right | = O(n^{-1/2}),
\end{equation}
where $ (Z_1, Z_2) \sim N_2(0, \Sigma)$ and $(X_j, 1 \le j \le m) \sim N_m(0, \Gamma)$.
From \eqref{conv:berryesseen2}, it follows that as $n \to \infty$,
\[   \PP\{ W'_1 >  a  | W_1 \ge   \tfrac{2 a}{\rho}  \}  \to \PP\{ Z_2 >  a  | Z_1  \ge  \tfrac{2a}{\rho} \} \]
and
\[  \PP\{ W'_j  \ge  a   | W_j \le 0 \} \to  \PP\{ Z_2  \ge  a  | Z_1 \le  0 \}.
\]
Recall that the conditional distribution  $Z_2 $ given  $Z_1$  is  $N(\rho Z_1, \sigma_{2.1}^2)$ where  $\sigma_{2.1}^2 = (1-\rho^2) m^{-1}(1 - m^{-1}) \le m^{-1}$. Also recall that  if $N$ is a standard normal random variable, then \[  \PP\{ N > x\} \le x^{-1} e^{-x^2} \quad \text{ for } x > 0,\]
and this bound is sharp in the asymptotic sense
\[ \PP\{ N > x\}  = \Theta(x^{-1} e^{-x^2/2}) \text{ as } x \to \infty. \]
Now for any $m \ge 2$,
\begin{align*}
\PP \left \{ Z_2 >   a \Big | Z_1  \ge  \tfrac{2 a}{\rho }  \right \}  &\ge  \PP \left  \{ \tfrac{Z_2  - \rho Z_1}{\sigma_{2.1}}  >  -\tfrac{a}{\sigma_{2.1} } \Big| Z_1  >  \tfrac{2a}{\rho}  \right \} \\
&\ge \PP\left \{ N >  - \sqrt{2\log (km)} \right \} \ge 1  - \tfrac{1}{  mk}.
\end{align*}
Similarly,
\[  \PP\left \{ Z_2  \ge a \Big| Z_1 \le  0 \right\}  \le \tfrac{1}{mk}. \]
Therefore,
\[ \lim_n \PP\{ W'_1  \ge  \max_{ j\ne 1} W'_j | A\} \ge 1 - \tfrac{1}{k }.  \]
Consequently, for $k \ge 2$, we have
\begin{align*}
 \lim_n \mathcal{M}_\rho(k, n) &\ge  \left (1 - \tfrac{1}{k}\right)^k \liminf_n \PP\{ A\}\\
 &\ge \tfrac14 \PP\big\{ X_1 \ge  \tfrac{2a}{\rho} \text{ and } X_j \le  0 \text{ for all } 2 \le j \le m \big\}  \ [\text{by } \eqref{conv:berryesseenm}].
 \end{align*}
 The proof of the lower bound is now complete by Lemma~\ref{l:mult_normal_estimate}.
 \end{proof}

 \begin{lemma} \label{l:mult_normal_estimate}
Fix $\rho \in (0, 1)$. Let $(X_j, 1 \le j \le m) \sim N_m(0, m^{-1} I_m  - m^{-2}\1_m\1_m' )$ and  $a = a(k, m) = \tfrac{\sqrt{2 \log(km)}}{\sqrt{m}}$. Then there exists $\gamma_2 = \gamma_2(\rho)>0$ such that for $k \ge 2$,
 \[ \PP\{ X_1 \ge \tfrac{2a}{\rho}  \text{ and } X_j \le  0 \text{ for all } 2 \le j \le m \} \ge c_2(m) k^{-\gamma_2}, \]
 where $c_2(m) \to 0$ as $m \to \infty$.
\end{lemma}
\begin{proof}[Proof of Lemma~\ref{l:mult_normal_estimate}] Note that $X_1 + X_2+ \ldots + X_m = 0$ with probability one. Therefore,
\begin{align}
\PP\{ X_1 \ge \tfrac{2a}{\rho}, X_j \le  0 \ \forall  j \ge 2 \}
&= \PP\{ \sum_{j=2}^m X_j \le -\tfrac{2a}{\rho},  X_j \le  0 \ \forall  j \ge 2\} \notag\\
 &\ge \PP\{ X_2 \le -\tfrac{2a}{\rho},   -\tfrac{1}{m^{3/2}} \le X_j \le  0 \ \forall  j \ge 3\} \label{eq:2mineq}.
\end{align}

 The conditional distribution of $X_2$ given $X_j, j \ge 3$ is given by \[ N \left(  - \tfrac{1}{2} \sum_{j=3}^m X_j, \frac{1}{2m}\right).\]
 Hence, it can be easily seen that
 \begin{align*}
 \eqref{eq:2mineq} &\ge  \PP\{ X_2  + \tfrac{1}{2} \sum_{j=3}^m X_j \le -\tfrac{2a}{\rho} - \tfrac{1}{\sqrt{m}} \} \PP\{  -\tfrac{1}{m^{3/2}}  \le X_j \le 0, j \ge 3 \}\\
 &\ge  \PP\{N \le -\tfrac{4 \sqrt{ \log(km)} }{\rho} - \sqrt{2} \} \PP\{ -\tfrac{1}{m^{3/2}}  \le X_j \le 0, j \ge 3 \} \\
 \end{align*}
where $N \sim N(0, 1)$. The lemma now follows from the normal tail estimate.
\end{proof}

\end{document}